\documentclass[a4paper,10pt]{amsart}

\usepackage[utf8]{inputenc} % direct use of umlauts

\usepackage{amsmath}
\usepackage{amstext}
\usepackage{amsfonts}
\usepackage{amsthm}
\usepackage{amssymb}
\usepackage{enumitem}

\usepackage{graphicx}
\usepackage{dsfont}
\usepackage{nicefrac}

\usepackage{hyperref}
\hypersetup{colorlinks}

\numberwithin{equation}{section}

\newcommand{\dualV}[2]{\langle #1, #2 \rangle_{V^*\times V}}

\newcommand{\dualVi}[2]{\langle #1, #2 \rangle_{V_{\xi_i}^*\times V_{\xi_i}}}
\newcommand{\dualX}[3][]{\langle #2,#3 \rangle_{#1^{*}\times #1}}
\newcommand{\inner}[3][]{( #2 , #3 )_{#1}}
\newcommand{\innerb}[3][]{\big(#2,#3\big)_{#1}}

\newcommand{\Vn}{V_{\xi_n}}
\newcommand{\VnOm}{V_{\xi_n(\omega)}}
\newcommand{\Vi}{V_{\xi_i}}

\newcommand{\mui}{\mu_{\xi_i}}

\newcommand{\kapi}{\kappa_{\xi_i}}

\newcommand{\lami}{\lambda_{\xi_i}}

\newcommand{\etan}{\eta_{\xi_n}}
\newcommand{\etai}{\eta_{\xi_i}}
\newcommand{\An}{A_{\xi_n}(t_n) }
\newcommand{\AnOm}{A_{\xi_n(\omega) }(t_n) }
\newcommand{\Ai}{A_{\xi_i}(t_i) }
\newcommand{\fn}{f^n }
\newcommand{\fii}{f^i }

\newcommand{\diff}[1]{\,\mathrm{d}#1}
\newcommand{\R}{\mathbb{R}}
\newcommand{\N}{\mathbb{N}}
\newcommand{\E}{\mathbb{E}}
\newcommand{\incl}{\hookrightarrow}

\newcommand{\F}{\mathcal{F}}
\renewcommand{\P}{\mathcal{P}}

\newcommand{\D}{\mathcal{D}}

\theoremstyle{plain}

\newtheorem{definition}{Definition}[section]
\newtheorem{theorem}[definition]{Theorem}
\newtheorem{lemma}[definition]{Lemma}

\newtheorem{assumption}{Assumption}

\theoremstyle{definition}
\newtheorem{remark}[definition]{Remark}

\begin{document}

\title[Randomized operator splitting]{A randomized operator splitting scheme inspired by stochastic optimization methods}

\author[M.~Eisenmann]{Monika Eisenmann}
\email{monika.eisenmann@math.lth.se}

\author[T.~Stillfjord]{Tony Stillfjord}
\email{tony.stillfjord@math.lth.se}

\address{  Centre for Mathematical Sciences\\
  Lund University\\
  P.O.\ Box 118\\
  221 00 Lund, Sweden}


\thanks{
This work was partially supported by the Crafoord foundation through the grant 
number 
20220657 and by the Wallenberg AI, Autono\-mous Systems and Software Program 
(WASP) funded by the Knut and Alice Wallenberg Foundation. 
The computations were enabled by resources provided by the Swedish National 
Infrastructure for Computing (SNIC) at LUNARC partially funded by the Swedish 
Research Council through grant agreement no. 2018–05973.	
}

\begin{abstract}
  In this paper, we combine the operator splitting methodology for abstract evolution 
  equations with that of stochastic methods for large-scale optimization problems. The 
  combination results in a randomized splitting scheme, which in a given time step 
  does not necessarily use all the parts of the split operator. This is in contrast to 
  deterministic splitting schemes which always use every part at least once, and often 
  several times. As a result, the computational cost can be significantly decreased in 
  comparison to such methods.
  We rigorously define a randomized operator splitting scheme in an abstract setting 
  and provide an error analysis where we prove that the temporal convergence order of the scheme 
  is at least $1/2$. We illustrate the theory by numerical experiments on both 
  linear and quasilinear diffusion problems, using a randomized domain decomposition 
  approach. We conclude that choosing the randomization in certain ways may 
  improve the order to $1$. This is as accurate as applying e.g.\ backward (implicit) 
  Euler to the full problem, without splitting.
\end{abstract}

\maketitle

\section{Introduction}

The main objective of this paper is to combine two successful strategies from 
the literature: the first being operator splitting schemes for evolution equations on
general, infinite dimensional frameworks and the second being stochastic optimization methods. 
Operator splitting schemes are an established tool in the field of numerical analysis of 
evolution equations and have a wide range of applications. Stochastic optimization methods have proven to be efficient at solving large-scale optimization problems, where it is infeasible to evaluate full gradients. They can drastically decrease the computational cost in e.g.\ machine learning settings.
The link between these two seemingly disparate areas is that an iterative method applied to an optimization problem can also be seen as a time-stepping method applied to a gradient flow connected to the optimization problem. In particular, stochastic optimization methods can then be interpreted as randomized operator splitting schemes for such gradient flows.
In this context, we introduce a general randomized splitting method that can be applied directly to evolution equations, and provide a rigorous convergence analysis.

Abstract evolution equations of the type
\begin{align*}
  \begin{cases}
    u'(t) + A(t)u(t) = f(t), \quad t \in (0,T],\\
    u(0) = u_0
  \end{cases}
\end{align*}
are an important building block for modeling processes in physics, biology and social sciences. Standard examples which appear in a variety of applications are fluid flow problems, where we model how a flow evolves on a given 
domain over time, compare \cite{Aronsson.1996, Kuijper.2007} and 
\cite[Section~1.3]{Vazquez.2007}. The operator $A(t)$ can denote, for example, a 
non-linear diffusion operator such as the $p$-Laplacian or a porous medium operator.

Deterministic operator splitting schemes as discussed 
in more detail in \cite{HundsdorferVerwer.2003} are a powerful tool for this type of 
equation. An example is given by a domain 
decomposition scheme, where we split the domain into sub-domains. Instead of solving 
one expensive problem on the entire domain, we deal with cheaper problems on the 
sub-domains. This is particularly useful in modern computer architectures, as the 
sub-problems may often be solved in parallel. 

Moreover, evolution equations are tightly connected to unconstrained optimization problems,
because the solution of $\min_u F(u)$ is a stationary point of the gradient flow
$u'(t) = -\nabla F(u(t))$. The latter is an evolution equation on an infinite time horizon
with $A = -\nabla F$ and $f = 0$.
In the large-scale case, such optimization problems benefit from stochastic optimization 
schemes. The most basic such method, the stochastic gradient descent, was first introduced already in~\cite{RobbinsMonro.1951}, but since then it has been extended and generalized in many
directions. See, e.g., the review article~\cite{BottouEtAll.2018} and the references therein.

Via the gradient flow interpretation, we can see these optimization methods as time-stepping
schemes where a randomly chosen sub-problem is considered in each time step. In 
essence, it
is therefore a \emph{randomized} operator splitting scheme. The difference between 
the works
mentioned above and ours is that we apply these stochastic optimization techniques to solve
the evolution equation itself rather than just finding its stationary state.

We consider nonlinear evolution equations in an abstract framework similar to 
\cite{EisenmannHansen.2022, Emmrich.2009, EmmrichThalhammer.2010} where 
operators of a monotone type have been studied. Deterministic splitting schemes for such
equations has been considered in 
e.g.~\cite{HansenOstermann.2010,HansenStillfjord.2013,JakobsenKarlsen.2005,LionsMercier.1979}.
 A particular kind of splitting schemes
which is most closely related to our work, domain decomposition methods, have been studied in 
\cite{EisenmannHansen.2018, EisenmannHansen.2022, HansenHenningsson.2017, 
Mathew.2008, Mathew.1998}. In this paper, we extend this framework of deterministic 
splitting schemes to a setting of randomized methods.

Outside of the context of optimization, other kinds of randomized methods have 
already proved themselves to be useful for solving evolution equations. Starting in  
\cite{Stengle.1990, Stengle.1995} explicit schemes for ordinary differential equations have 
been randomized. This approach has been further extended in \cite{BochacikEtAl.2021, 
Daun.2011, JentzenNeuenkirch.2009, KruseWu.2017, KruseWu.2019}. 
In \cite{EKKL.2019}, it has been extended both to implicit methods and to 
partial differential equations and in \cite{KruseWu.2018} to finite element approximations. 
While these works considered certain randomizations in their schemes, they are 
conceptually different from our approach. Their main idea is to approximate any appearing integrals through
\begin{align*}
  \int_{t_{n-1}}^{t_n} f(t) \diff{t} \approx f(\xi_n)
  \quad \text{and} \quad 
  \int_{t_{n-1}}^{t_n} A(t)v \diff{t} \approx A(\xi_n) v,
\end{align*}
where $\xi_n$ is a random variable that takes on values in $[t_{n-1}, t_n]$. This ansatz
coincides with a Monte Carlo integration idea. 
In this paper, we use a different approach where we decompose the operator in a 
randomized fashion. More precisely, we approximate data
\begin{align*}
  f = \frac{1}{s}\sum_{\ell = 1}^{s} f_{\ell}
  \quad \text{and} \quad 
  A = \frac{1}{s}\sum_{\ell = 1}^{s} A_{\ell}
\end{align*}
by
\begin{align*}
  f_B = \frac{1}{|B|}\sum_{\ell \in B} f_{\ell}
  \quad \text{and} \quad 
  A_B = \frac{1}{|B|}\sum_{\ell \in B} A_{\ell}
\end{align*}
where the batch $B \subset \{1,\dots,s\}$ is chosen randomly. The stochastic 
approximations $f_B$ and $A_B$ of the original data $f$ and $A$ are cheaper to evaluate 
in applications.
This is less related to Monte Carlo integration and more similar to 
stochastic optimization methods, compare~\cite{BottouEtAll.2018, EisenmannEtAl.2022}.
Similar ideas have been considered in~\cite{JinEtAl.2020,JinEtAl.2021, LiEtAl.2020}, where
a random batch method for interacting particle systems has been studied. Moreover, very
recently and during the preparation of this work, a similar approach has also been applied
to the optimal control of linear time invariant (LTI) dynamical systems in~\cite{VeldmanZuazua.2022}.
While the convergence rate provided there is essentially the same as what we establish in
our main result Theorem~\ref{thm:convH}, our setting is more general and allows for nonlinear
operators on infinite dimensional spaces rather than finite dimensional matrices. We also consider the error of the time stepping method that is used to approximate the solution to $u'(t) + A_B(t)u(t) = f_B(t)$, while the error bounds in~\cite{VeldmanZuazua.2022} assume that this evolution equation is solved exactly.

This paper is organized as follows. In Section~\ref{sec:setting}, we begin by explaining 
our abstract framework. This includes both the precise assumptions that we make and 
the definition of our time-stepping scheme. We give a more concrete application of the 
abstract framework in Section~\ref{sec:example}. With the setting fixed, we first prove in Section~\ref{sec:wellDef} that the scheme and its solution are indeed well-defined.
We prove the convergence of the scheme in expectation in Section~\ref{sec:convH}. These
theoretical convergence results are illustrated by numerical experiments with 
two-dimensional linear and quasilinear
nonlinear and linear diffusion problem in Section~\ref{sec:numExperiments}. Finally, we collect some
more technical auxiliary results in Appendix~\ref{sec:appendix}.  

\section{Setting} \label{sec:setting}

In the following, we introduce a theoretical framework for the randomized operator 
splitting. This setting is similar to the one in~\cite{EisenmannHansen.2022}. 

\begin{assumption}\label{ass:spaces}
  Let $(H,\inner[H]{\cdot}{\cdot},\|\cdot\|_H)$ be a real, separable Hilbert space and let $(V, 
  \|\cdot\|_V)$ be a real, separable, reflexive Banach space, which is continuously and 
  densely embedded into $H$. 
  Moreover, there exists a semi-norm $|\cdot|_V$ on $V$.
\end{assumption}

Denoting the dual space of $V$ by $V^*$ and identifying the Hilbert space $H$ with its 
dual space, the spaces from Assumption~\ref{ass:spaces} form a Gelfand triple and 
fulfill, in particular,
\begin{align*}
  V \overset{d}{\incl} H \cong H^* \overset{d}{\incl}  V^*.
\end{align*}

\begin{assumption}\label{ass:A}
  Let the spaces $H$ and $V$ be given as stated in  
  Assumption~\ref{ass:spaces}.
  Furthermore, for $T \in (0,\infty)$ as well as $p \in [2,\infty)$, let 
  $\{A(t)\}_{t \in [0,T]}$ be a family of operators $A(t) \colon V \to V^*$ that
  satisfy the following conditions:
  \begin{enumerate}
    \item[(i)] The mapping $Av \colon [0,T] \to V^*$ given by $t \mapsto 
    A(t)v$ is continuous almost everywhere in $(0,T)$ for all $v \in V$.
    \item[(ii)] The operator $A(t) \colon V \to V^*$, $t \in [0,T]$, is    
    radially continuous, i.e., the mapping $s \mapsto \dualV{A(t)(v+s w)}{w}$ 
    is continuous on $[0,1]$ for all $v,w \in V$.
    \item[(iii)] For $\kappa_A \in [0,\infty)$, the operator $A(t) + \kappa_A I \colon V \to 
    V^*$, $t \in [0,T]$, fulfills a monotonicity-type condition in the sense that there exists 
    $\eta_A \in [0,\infty)$, which does not depend on $t$, such that
    \begin{align*}
      \dualV{A(t)v - A(t)w}{v - w} + \kappa_A \|v - w\|_H^2 \geq \eta_A | v - w |_V^p
    \end{align*}
    is fulfilled for all $v,w \in V$.
    \item[(iv)] The operator $A(t) \colon V \to V^*$, $t \in [0,T]$, is uniformly 
    bounded such that there exists $\beta_A \in [0,\infty)$, which does not 
    depend on $t$, with
    \begin{align*}
      \|A(t) v \|_{V^*} \leq \beta_A \big(1 + \|v\|_V^{p-1}\big)
    \end{align*}
    for all $v \in V$.
    \item[(v)] The operator $A(t) + \kappa_A I \colon V \to V^*$, $t \in [0,T]$, fulfills a     
    uniform semi-coercivity condition such that there exist $\mu_A, \lambda_A  \in 
    [0,\infty)$, which do not depend on $t$, with
    \begin{align*}
      \dualV{A(t) v}{v}+ \kappa_A \|v\|_H^2  + \lambda_A \geq  \mu_A |v|_V^p
    \end{align*}
    for all $v \in V$.
  \end{enumerate}
\end{assumption}

\begin{assumption}\label{ass:u0F}
  The function $f$ is an element of the Bochner space $L^2(0,T;H)$, and the initial value $u_0 \in H$, where $H$ is the Hilbert space from Assumption~\ref{ass:spaces}.
\end{assumption}

Assumption~\ref{ass:spaces}--\ref{ass:u0F}, are requirements on the problem that 
we want to solve. The following 
Assumptions~\ref{ass:xis}--\ref{ass:expectation} 
are needed to state the approximation scheme for the given problem.

\begin{assumption}\label{ass:xis}
  Let $(\Omega, \F, \P)$ be a complete probability space and let $\{\xi_n\}_{n \in \N}$ 
  be a family of mutually independent random variables.   
  Further, let the filtration $\{\F_n\}_{n \in \N}$ be given by
  \begin{align*}
    &\F_0 := \sigma\big(\mathcal{N} \in \F : \P(\mathcal{N}) = 0 \big)\\
    &\F_n := \sigma\big(\sigma\big(\xi_i : i \in \{1,\dots,n\}\big) \cup \F_0 \big),
    \quad n \in \N,
  \end{align*}
  where $\sigma$ denotes the generated $\sigma$-algebra.
\end{assumption}  

In the following, we denote the expectation with respect to the probability distribution 
of $\xi$ for a random variable $X$ in the Bochner space $L^1(\Omega; H)$ by 
$\E_{\xi}[X]$. Moreover, 
we abbreviate the total expectation by
\begin{align*}
  \E_n [X] = \E_{\xi_1}[\E_{\xi_2}[ \dots \E_{\xi_n}[X] \dots]].
\end{align*}
We denote the space of H\"older continuous functions on $[0,T]$ with H\"older 
coefficient $\gamma \in (0,1)$ and values in $H$ by $C^{\gamma}([0,T];H)$. For 
notational convenience we include the case $\gamma = 1$ and denote the space of 
Lipschitz continuous functions by $C^{1}([0,T];H)$.

\begin{assumption}\label{ass:expectation}
  Let Assumptions~\ref{ass:spaces}--\ref{ass:xis} be fulfilled.
  Assume that for almost every $\omega \in \Omega$, there exists a real Banach 
  space $V_{\xi(\omega)}$ such that $V \stackrel{d}{\incl} V_{\xi(\omega)} 
  \stackrel{d}{\incl} H$, $\bigcap_{\omega \in \Omega} V_{\xi(\omega)} = V$ and there 
  exists a  semi-norm $|\cdot|_{V_{\xi(\omega)}}$. 
  Moreover, the mapping from $\omega \mapsto V_{\xi(\omega)}$ is measurable in the 
  sense that for every $v \in H$ the set $\{ \omega \in \Omega: v \in V_{\xi(\omega)}\}$ 
  is an element of the complete generated $\sigma$-algebra 
  \begin{align*}
    \F_{\xi} := \sigma\big(\sigma (\xi ) \cup \sigma\big(\mathcal{N} \in \F : 
    \P(\mathcal{N}) = 0 \big) \big).
  \end{align*}
  Further, let the family of operators $\{A_{\xi(\omega)}(t)\}_{\omega \in \Omega, t \in 
  [0,T]}$ be such that for almost every $\omega \in \Omega$,  $\{A_{\xi(\omega)}(t)\}_{t 
  \in [0,T]}$ fulfills Assumption~\ref{ass:A} with the spaces $V_{\xi(\omega)}$, $H$ and 
  $V_{\xi(\omega)}^*$ and corresponding constants $\kappa_{\xi(\omega)}$, 
  $\eta_{\xi(\omega)}$, $\beta_{\xi(\omega)}$, $\mu_{\xi(\omega)}$ and 
  $\lambda_{\xi(\omega)}$. 
  Moreover, the mapping $A_{\xi}(t) v \colon \Omega \to V^*$ is $\F_{\xi}$-measurable 
  and the equality $\E_{\xi} [ A_{\xi}(t) v ] = A(t) v$ is fulfilled in $V^*$ for $v \in V$.
  The mappings $\kappa_{\xi}, \eta_{\xi}, \mu_{\xi}, \beta_{\xi}, \lambda_{\xi} \colon 
  \Omega \to [0,\infty)$ are measurable and there exist $\kappa, \lambda \in [0,\infty)$ 
  which fulfill $\kappa_{\xi} \leq \kappa$ almost surely and $\E_{\xi} 
  \big[\lambda_{\xi}  \big] \leq \lambda$.
  
  Further, let the family $\{f_{\xi(\omega)}\}_{\omega \in \Omega}$ be given such that 
  $f_{\xi(\omega)} \in L^2(0,T; H)$. Moreover, the mapping $f_{\xi}(t) \colon \Omega \to 
  H$ is $\F_{\xi}$-measurable and $\E_{\xi} [ f_{\xi}(t) ] = f(t)$ is fulfilled in 
  $H$ for almost all $t \in (0,T)$.
\end{assumption}

Under the setting explained in the above assumptions, we consider the initial value 
problem
\begin{align}
  \label{eq:equation}
  \begin{cases}
    u'(t) + A(t)u(t) = f(t)\quad&\text{in } V^*, \quad t \in (0,T],\\
    u(0) = u_0 &\text{in } H.
  \end{cases}
\end{align}
For a non-uniform temporal grid $0 = t_0 <t_1 < \dots < t_N = T$, a step size $h_n = 
t_n - t_{n-1}$, $h = \max_{n \in \{1,\dots,N\}} h_n$, and a family of random variables 
$\{f^n\}_{n \in \{1,\dots,N\}}$ such that $f^n\colon \Omega \to H$ is 
$\F_{\xi_n}$-measurable, we consider the scheme
\begin{align}
  \label{eq:scheme}
  \begin{cases}
    U^n - U^{n-1} + h_n \An U^n = h_n \fn \quad &\text{in } V_{\xi_n}^*,  \quad n \in 
    \{1,\dots,N\},\\
    U^0 = u_0 &\text{in } H.
  \end{cases}
\end{align}
Note that $U^n \colon \Omega \to H$ is a random variable and therefore some 
statements involving it below only hold almost surely. Whenever there is no risk of 
misinterpretation, we 
omit writing almost surely for the sake of brevity.

When proving that the scheme is well-defined and establishing an a priori bound, it is 
sufficient to assume that $\{f_{\xi_n}\}_{n \in \{1,\dots,N\}}$ are integrable with respect 
to 
the temporal parameter. In that case, we can choose for example
\begin{align}
  \label{eq:randomData1}
  \fn = \frac{1}{h_n} \int_{t_{n-1}}^{t_n} f_{\xi_n}(t) \diff{t} \quad \text{in } H
  \text{ almost surely.}
\end{align}
In our error bounds, we assume more regularity for the functions $\{f_{\xi_n}\}_{n \in 
\{1,\dots,N\}}$ and demand continuity with respect to the temporal parameter. In this 
case, we may also use 
\begin{align}
  \label{eq:randomData2}
  \fn = f_{\xi_n}(t_n) \quad \text{in } H \text{ almost surely.}
\end{align}
We will focus on this second choice for the error bounds in Section~\ref{sec:convH}.

\section{Application: Domain decomposition}
\label{sec:example}

One main application that is allowed by our abstract framework is a domain 
decomposition scheme for a nonlinear fluid flow problem. Domain decomposition 
schemes are well-known for deterministic operator splittings. However,  to the best of 
our knowledge, it has not been studied in the context of a randomized operator 
splitting scheme.

\subsection{Deterministic domain decomposition}
To exemplify our abstract equation \eqref{eq:equation}, we consider a (nonlinear) parabolic differential 
equation. In the following, let $\D \subset \R^d$, $d \in \N$, be a bounded
domain with a Lipschitz boundary $\partial \D$. For $p \in [2, \infty)$, we 
consider the parabolic $p$-Laplacian with homogeneous Dirichlet boundary conditions
\begin{align}\label{eq:parabolicPDE}
  \begin{cases}
    \partial_t u(t,x) - \nabla \cdot (\alpha(t) |\nabla u(t,x)|^{p-2}\nabla u(t,x))  = 
    \tilde{f}(t,x), \quad &(t,x) \in (0,T) \times \D,\\
    u(t,x) = 0,
    &(t,x) \in (0,T) \times \partial \D,\\
    u(0,x) = u_0(x), &x \in \D,
  \end{cases}
\end{align}
for $\alpha \colon [0,T] \to \R$ and $u_0 \colon \D \to \R$. The notation $\tilde{f}$ is used to differentiate between the function $\tilde{f} \colon (0,T) \times  \D \to \R$ and the abstract function $f$ on $(0,T)$ that it gives rise to through $[f(t)](x) = \tilde{f}(t,x)$.
We consider a domain decomposition 
scheme similar to \cite{HansenHenningsson.2017} for $p = 2$ and to \cite{EisenmannHansen.2018, 
EisenmannHansen.2022} for $p \in [2,\infty)$. For the sake of completeness, we 
recapitulate the setting here also with a different boundary condition. 

For $s \in \N$, let $\{ \D_{\ell} \}_{\ell =1}^{s}$ be a family of overlapping 
subsets of $\D$. Let each subset have a Lipschitz boundary and let the union
of them fulfill $\bigcup_{\ell =1}^s \D_{\ell} = \D$.
On the sub-domains $\{ \D_{\ell} \}_{\ell =1}^{s}$, let the partition of unity 
$\{\chi_{\ell} \}_{\ell =1}^{s}\subset W^{1,\infty}(\D)$ be given such that the following 
criteria are fulfilled
\begin{align*}
  \chi_{\ell} (x)>0\text{ for all }x\in\D_{\ell},
  \quad
  \chi_{\ell} (x) = 0\text{ for all }x\in\D\setminus\D_{\ell},
  \quad
  \sum_{\ell =1}^{s} \chi_{\ell}= 1
\end{align*}
for $\ell \in \{1,\dots,s\}$. 
With the help of the functions $\{\chi_{\ell}\}_{\ell \in \{1,\dots,s\}}$, it is now possible to 
introduce suitable 
functional spaces $\{V_{\ell}\}_{\ell \in \{1,\dots,s\}}$. We use the weighted Lebesgue 
space $L^p(\D_{\ell},\chi_{\ell})^d$ that consists of all measurable functions $v = 
(v_1,\dots,v_d) \colon \D_{\ell} \to \R^d$ such that
\begin{align*}
  \|(v_1,\ldots,v_{d})\|_{L^p(\D_{\ell},\chi_{\ell})^d}
  = \Big(\int_{\D_{\ell}}\chi_{\ell} |(v_1,\ldots,v_{d})|^p 
  \diff{x}\Big)^{\frac{1}{p}}
\end{align*}
is finite. 
In the following, let the pivot space $\left(H, \inner[H]{\cdot}{\cdot}, \|\cdot\|_H \right)$ be 
the space $L^2(\D)$ of square integrable functions on $\D$ with the usual norm 
and inner product. The spaces $V$ and $V_{\ell}$, $\ell \in \{1,\dots,s\}$, are given by
\begin{align*}
  V = \text{clos}_{\|\cdot \|_{V}} \big(C_0^{\infty}(\D)\big) = W_0^{1,p}(\D) \quad \text{and} 
  \quad 
  V_{\ell} 
  = \text{clos}_{\|\cdot \|_{V_{\ell}}} \big(C_0^{\infty}(\D)\big), 
\end{align*}
with respect to the norms
\begin{align}
  \label{eq:normsV} 
  \|\cdot\|_{V}
  = \|\cdot\|_H + \|\nabla \cdot\|_{L^p(\D)^d}
  \quad\text{and}\quad
  \|\cdot\|_{V_{\ell}}
  = \|\cdot\|_H + \| \nabla \cdot \|_{L^p(\D_{\ell},\chi_{\ell})^d}
\end{align}
and semi-norms
\begin{align*}
  |\cdot |_{V}
  = \|\nabla \cdot\|_{L^p(\D)^d}
  \quad\text{and}\quad
  |\cdot|_{V_{\ell}}
  = \| \nabla \cdot \|_{L^p(\D_{\ell},\chi_{\ell})^d}.
\end{align*}
Note that a bootstrap argument involving the Sobolev embedding theorem 
shows that 
the norm given in \eqref{eq:normsV} is equivalent to the standard norm in the space.
We can now introduce the operators $A(t) \colon V \to V^*$, 
$A_{\ell}(t) \colon V_{\ell} \to V^*_{\ell}$, $\ell\in \{ 1,\dots,s\}$, $t\in [0,T]$, given by
\begin{align*} 
  \dualV{A(t) u}{v} &= \int_{\D} \alpha(t) |\nabla u|^{p-2} \nabla u \cdot 
  \nabla v \diff{x},
  \quad u,v\in V, \\
  \dualX[V_{\ell}]{A_{\ell}(t) u}{v}
  &= \int_{\D_{\ell}} \chi_{\ell} \alpha(t) |\nabla u|^{p-2} \nabla u \cdot 
  \nabla v \diff{x},
  \quad u,v\in V_{\ell}.
\end{align*}
Similarly, we define the right-hand sides $f_{\ell} \colon [0,T] \to H$, $\ell \in 
\{1,\dots,s\}$, where $f_{\ell}(t) = \chi_{\ell} f(t)$ in $H$ for almost every $t \in (0,T)$.

\begin{lemma} \label{lem:example}
  Let the parameters of the equation \eqref{eq:parabolicPDE} be given such that $\alpha 
  \in C([0,T];\R)$, $u_0 \in L^2(\D)$ and $\tilde{f} \in L^2((0,T) \times \D)$.
  Then the setting described above fulfills Assumptions~\ref{ass:spaces}--\ref{ass:u0F}.

  Let the partition of unity $\{\chi_{\ell} \}_{\ell =1}^{s}\subset W^{1,\infty}(\D)$ 
  fulfill that for every function $\chi_{\ell}$ there exists $\varepsilon_0 \in (0,\infty)$ 
  such that $\D_{\ell}^{\varepsilon} = \{ x\in \D_{\ell} : \chi_{\ell}(x) \geq \varepsilon \}$
  is a Lipschitz domain for all $\varepsilon \in (0,\varepsilon_0)$. 
  Then $V$ and $V_{\ell}$, $\ell \in \{1,\dots,s\}$, are reflexive Banach spaces and $V = 
  \bigcap_{\ell = 1}^s V_{\ell}$.
  Further, the family of operators $\{A_{\ell}(t)\}_{t \in [0,T]}$, $\ell \in 
  \{1,\dots,s\}$ fulfills 
  Assumption~\ref{ass:A} with the spaces $V_{\ell}$, $H$ and $V_{\ell}^*$. 
  Moreover, $\sum_{\ell = 1}^{s} A_{\ell}(t) v = A(t) v$ is fulfilled in $V^*$ for $v \in V$ 
  for almost every $t \in (0,T)$ and 
  corresponding constants $\kappa_A = \kappa_{\ell} = \lambda_A = \lambda_{\ell} 
  = 0$, $\mu_A = \mu_{\ell} = \eta_A = \eta_{\ell} = 1$.
  
  Finally, the family $\{f_{\ell}\}_{\ell \in \{1,\dots,s\}}$ fulfills
  $f_{\ell} \in L^2(0,T; H)$ and $\sum_{\ell = 1}^{s} f_{\ell}(t) = f(t)$ in 
  $H$ for almost all $t \in (0,T)$. 
\end{lemma}

\begin{proof}
  The space $H = L^2(\D)$ is a real, separable Hilbert space, while $V = 
  W_0^{1,p}(\D)$ is a real, separable Banach space that is densely embedded into 
  $H$. Thus, they fulfill Assumption~\ref{ass:spaces}. 
  Analogously to \cite[Lemma~3]{EisenmannHansen.2018}, the spaces $V$ and 
  $V_{\ell}$, $\ell \in \{1,\dots,s\}$, are reflexive Banach spaces and since 
  $C_0^{\infty}(\D)$ is dense in $H$ and $C_0^{\infty}(\D) \subseteq V \subset 
  V_{\ell}$ it follows that $V$ and $V_{\ell}$ are dense in $H$.
  It remains to prove that $\bigcap_{\ell = 1}^s V_{\ell} = V$ is fulfilled. First, we notice 
  that $\|w\|_{L^p( \D_{\ell},\chi_{\ell})^d} \leq \|w\|_{L^p(\D)^d}$ for every $w \in 
  L^p(\D)^d$. Thus, it follows that $V \subseteq V_{\ell}$ for every $\ell \in 
  \{1,\dots,s\}$ and in particular $V \subseteq \bigcap_{\ell = 1}^s V_{\ell}$. The other 
  inclusion $\bigcap_{\ell = 1}^s V_{\ell} \subseteq V$ requires more attention. 
  For $\varepsilon \in (0,\infty)$, we introduce the set $\D_{\ell}^{\varepsilon}
  = \{ x \in \D : \chi_{\ell}(x) \geq \varepsilon \}$.
  By assumption the sets $\D_{\ell}^{\varepsilon}$ have Lipschitz boundary for 
  $\varepsilon$ small enough. We consider the spaces of restricted 
  functions
  \begin{align*}
    C_0^{\infty}(\D)|_{\D_{\ell}^{\varepsilon}}
    = \{ u \in C^{\infty}(\D_{\ell}^{\varepsilon})
    : u|_{\partial \D_{\ell}^{\varepsilon} \cap \partial \D} = 0\}
  \quad \text{and} \quad
    V_{\ell}^{\varepsilon}
    = \{ u|_{\D_{\ell}^{\varepsilon}} : u \in V_{\ell} \}.
  \end{align*}
  If a weight function $\chi_{\ell}$ fulfills $0 < \varepsilon < \chi_{\ell} \leq 
  1 <\infty$ on the whole domain $\D$, it follows that the weighted Lebesgue space 
  $L^p(\D_{\ell}^{\varepsilon},\chi_{\ell})^d$ coincides with the space 
  $L^p(\D_{\ell}^{\varepsilon})^d$ (see, e.g., \cite[Chapter~3]{Kufner.1980}). Thus, 
  we obtain 
  $V_{\ell}^{\varepsilon}= W^{1,p}(\D_{\ell}^{\varepsilon})$. The continuity of the trace 
  operator (see, e.g., \cite[Theorem~15.23]{Leoni.2009}), implies that
  \begin{align*}
    \overline{C_0^{\infty}(\D)|_{\D_{\ell}^{\varepsilon}}
    }^{\|\cdot\|_{V_{\ell} }}
    = \{ u \in W^{1,p}(\D_{\ell}^{\varepsilon}) :
    u|_{\partial \D_{\ell}^{\varepsilon} \cap \partial \D} = 0\}.
  \end{align*}
  This shows that $u \in V_{\ell}$ is zero on $\partial \D_{\ell}^{\varepsilon}
  \cap \partial \D$ for every $\varepsilon \in (0,\infty)$ small enough. As 
  $\varepsilon$ can be chosen arbitrarily small, it follows that $u \in V_{\ell}$ fulfills
  $v|_{\partial \D \cap \partial \D_{\ell}} = 0$.
  In combination with \cite[Lemma~1]{EisenmannHansen.2018}, we obtain that 
  $\bigcap_{\ell = 1}^{s} V_{\ell} = W^{1,p}_0(\D) = V$.
  
  Similar to the argumentation of \cite[Lemma~4]{EisenmannHansen.2018}, it follows that 
  the families of operators $\{A(t)\}_{t \in [0,T]}$ and $\{A_{\ell}(t)\}_{t \in [0,T]}$, 
  $\ell \in \{1,\dots,s\}$, fulfills Assumption~\ref{ass:A} with respect to the 
  corresponding spaces with $\kappa_A = \kappa_{\ell} = \lambda_A = \lambda_{\ell} 
  = 0$, $\mu_A = \mu_{\ell} = \eta_A = \eta_{\ell} = 1$.
  
  Assumption~\ref{ass:u0F} is fulfilled as $\tilde{f} \in L^2((0,T) \times \D)$
  means that the abstract function $f$ belongs to $L^2(0,T;L^2(\D))$. Thus, as 
  $\chi_{\ell} \in W^{1,\infty}(\D)$, it follows that $f_{\ell} = \chi_{\ell} f \in L^2(0,T;H)$ 
  and $\sum_{\ell = 1}^{s} f_{\ell}(t) = f(t)$ in $H$ for almost every $t \in (0,T)$.
\end{proof}

\subsection{Randomized scheme}

For a randomized splitting in combination with a domain decomposition, different 
approaches can be applied. One possibility is to choose a random support of the weight 
functions $\{\chi_{\ell}\}_{\ell \in \{1,\dots,s\}}$. This could possibly be done efficiently
using priority queue techniques similar to those in \cite{StoneGeigerLord.2017}.
In this paper, we instead fix the weight functions, but choose a random part 
of the operator in every time step. For the operator $A(t) = \sum_{\ell = 1}^{s} A_{\ell}(t)$ 
and a right hand side $f(t) = \sum_{\ell = 1}^{s} f_{\ell}(t)$, we 
introduce a random variable $\xi \colon \Omega \to 2^{\{1, \dots, s\}}$ such that 
$[A_{\xi}(t)](\omega) = \sum_{\ell \in \xi(\omega)} A_{\ell}(t) / \tau_{\ell}$ and 
$[f_{\xi}(t)](\omega) = \sum_{\ell \in \xi(\omega)} f_{\ell}(t) / \tau_{\ell}$ with
\begin{equation*}
  \tau_{\ell} = 
  \sum_{ B \in 2^{\{1, \dots, s\}} : \ \ell \in B} \P(\Omega_{\xi  = B})
  \quad \text{with} \quad \Omega_{\xi  = B} = \{ \omega \in \Omega : \xi(\omega) = B\}.
\end{equation*}
Here $\tau_{\ell}$ is the proper scaling factor which ensures that $\E_{\xi} [A_{\xi}(t)] = 
A(t)$ and $\E_{\xi} [f_{\xi}(t)] = f(t)$. We tacitly 
assume that $\tau_{\ell} > 0$, because otherwise we would be in a situation where at 
least one $A_{\ell}(t)$ is never chosen. Such a strategy would obviously not work.
We set $V_{\xi(\omega)} = \bigcap_{\ell \in \xi(\omega)} V_{\ell}$.

\begin{lemma}
  Let $\{\xi_n\}_{n \in \{1,\dots,N\}}$ fulfill Assumption~\ref{ass:xis} such that $\xi_n 
  \colon \Omega \to 2^{\{1,\dots,s\}}$ and $\xi_n^{-1}(B) \in \F_{\xi_n}$ for all $B 
  \subset 2^{\{1,\dots,s\}}$ and $n \in \{1,\dots,N\}$.
  Under the setting above, Assumption~\ref{ass:expectation} is fulfilled.
\end{lemma}

\begin{proof}
  In the following proof, we drop the index $n$ to keep the notation simpler.
  The embedding and norm properties are fulfilled as verified in the previous lemma. It 
  remains to verify the measurability condition.
  We need to verify that for every $v \in H$, the set $\{\omega \in \Omega: v \in 
  V_{\xi(\omega)}\} \in \F_{\xi} = \sigma \big( \sigma(\xi) \cup \sigma(\mathcal{N} \in \F : 
  \P(\mathcal{N}) = 0)\big)$. 
  For fixed $v \in H$, we set $B_v = \{\ell \in \{1,\dots, s\}: v \in V_{\ell}\} \in 
  2^{\{1,\dots,s\}}$. Then it follows that
  \begin{align*}
    \{\omega \in \Omega: v \in V_{\xi(\omega)}\}
    = \big\{ \omega \in \Omega: \xi(\omega) \in 2^{B_v} \big\} 
    = \xi^{-1}\big(2^{B_v}\big) \in \F_{\xi}.
  \end{align*}
  Moreover, we need to verify that the mapping $\omega \mapsto A_{\xi(\omega)}(t)v$ 
  is measurable for every $v \in H$. This can be seen from the decomposition
  $A_{\xi}(t)v = S_{A(t)v} \circ \xi$ where $S_{A(t)v} \colon 2^{\{1,\dots,s\}} \to V^*$
  is given through $S_{A(t)v} (B) = \sum_{\ell \in B} A_{\ell}(t)v$. As $\xi^{-1}(B) \in 
  \F_{\xi}$
  for all $B \subset 2^{\{1,\dots,s\}}$ and $S_{A(t)v}^{-1}(X) \subset  2^{\{1,\dots,s\}}$
  for any open set $X \subset V^*$, the mapping $\omega \mapsto A_{\xi(\omega)}(t)v$ 
  is measurable. Analogously, it can be proved that mapping $\omega \mapsto 
  f_{\xi(\omega)}(t)$ is measurable.
  In Lemma~\ref{lem:example}, we already verified that an operator $A_{\xi(w)}$ fulfills 
  the conditions from Assumption~\ref{ass:A}. Thus, it only remains to prove the  
  expectation property from Assumption~\ref{ass:expectation}. This is fulfilled as
  \begin{align*}
    \E_{\xi} [A_{\xi}(t) v]  
    &= \sum_{B \in 2^{\{1,\dots,s\}}} \P(\Omega_{\xi = B})
      \sum_{\ell \in B} \frac{1}{\tau_{\ell}} A_{\ell}(t) v \\
    &= \sum_{\ell = 1}^{s} \frac{1}{\tau_{\ell}}  A_{\ell}(t) v \sum_{B \in 2^{\{1,\dots,s\}} 
      : \ \ell  \in B }{ \P(\Omega_{\xi = B}) }
    = \sum_{\ell = 1}^{s}   A_{\ell}(t) v = A(t) v \quad \text{in } V^*
  \end{align*}
  holds true for $v \in V$ and for almost every $t \in [0,T]$. The same algebraic manipulation in $H$ instead of $V^*$ shows that $\E_{\xi} [f_{\xi}(t)] = f(t)$.
\end{proof}

\section{Solution is well-defined}
\label{sec:wellDef}

In the coming section, we show that our scheme \eqref{eq:scheme} is well-defined. This 
includes that first of all the scheme possesses a unique solution. We consider a purely 
deterministic equation~\eqref{eq:equation}. However, as the numerical scheme is 
randomized, the solution $U^n$ of \eqref{eq:scheme} is a mapping of the type $U^n 
\colon \Omega \to H$. 
Thus, we also need to make sure that it is a measurable function. These facts are verified 
in Lemma~\ref{lem:ExistenceMeasurability}. Moreover, we provide an integrability result 
in the form of an a priori bound in Lemma~\ref{lem:apriori}.

\begin{lemma} \label{lem:ExistenceMeasurability}
  Let Assumptions~\ref{ass:spaces}--\ref{ass:expectation} be fulfilled. 
  Further, let the random variables $\fn \colon \Omega \to H$ be given such that they are 
  $\F_{\xi_n}$-measurable for every $n \in \{1,\dots, N\}$. 
  Then for $\kappa h_n \leq \kappa h < 1$ there exists a unique $\F_n$-measurable 
  function $U^n \colon \Omega \to H$ such that $U^n(\omega) \in V_{\xi_n(\omega)}$ 
  and \eqref{eq:scheme} is fulfilled for every $n \in \{1,\dots,N\}$.
\end{lemma}

\begin{proof}
  For $\omega \in \Omega$, we find that the operator $I + h_n  \AnOm \colon 
  \VnOm \to \VnOm^*$ is monotone, radially continuous and coercive. Thus, it is 
  surjective, compare \cite[Theorem~2.18]{Roubicek.2013}. 
  Moreover, for $U_1, U_2 \in \VnOm$ with $\big(I + h_n  \AnOm\big)U_1 = \big(I + 
  h_n \AnOm\big)U_2$, it follows that
  \begin{align*}
    0&= \langle \big(I + h_n  \AnOm\big)U_1 - \big(I + h_n  
      \AnOm\big)U_2, U_1 - U_2 \rangle_{\VnOm^* \times \VnOm}\\
    &\geq \big(1 - h_n \kappa \big) \|U_1 - U_2 \|_H^2.
  \end{align*}
  Thus, it follows that $\|U_1 - U_2 \|_H = 0$ and $I + h_n  \AnOm$ is injective for 
  $\kappa h_n < 1$ and, in particular, bijective.
  
  It remains to verify that $U^n \colon \Omega \to H$ is well-defined. We define the 
  auxiliary function $g \colon \Omega \times H \to V^*$ such that
  \begin{align*}
    (\omega, U) \mapsto 
    \begin{cases}
      h_n \fn(\omega) + U^{n-1} - \big(I + h_n  \AnOm \big) U, &U \in 
      \VnOm\\
      e, &U \in H \setminus \VnOm,
    \end{cases}
  \end{align*}
  where $e \in V^*$ with $\|e\|_{V^*} = 1$. In the following, we want to apply 
  Lemma~\ref{lem:measurability} to the function $g$ to prove that $U^n$ is 
  measurable. Applying \cite[Lemma~2.16]{Roubicek.2013}, it follow that for fixed 
  $\omega \in \Omega$, the function $v \mapsto \dualV{g(\omega, v)}{w}$ is 
  continuous for all $v, w \in \VnOm$. 
  It remains to verify that for fixed $v \in H$ and $w \in V$, 
  the function $\omega \mapsto \dualV{g(\omega, v)}{w}$ is measurable.
  Let $B$ be an open set in $V^*$. It then follows that 
  \begin{align*}
    &\big(g(\cdot, v) \big)^{-1} (B) \\
    &= \{ \omega \in \Omega : g(\omega, v) \in B \}\\
    &= \{ \omega \in \Omega : v \in \VnOm, h_n \fn(\omega) + U^{n-1} - \big(I + 
    h_n  \AnOm \big) v \in B \}\\
    &\quad \cup \{ \omega \in \Omega : v \in H \setminus \VnOm, e \in B \} \\
    &= \big(\{ \omega \in \Omega : v \in \VnOm \} 
    \cap \{ \omega \in \Omega : h_n \fn(\omega) + U^{n-1} - \big(I + h_n  
    \AnOm 
    \big) v \in B \} \big)\\
    &\quad \cup \big(\{ \omega \in \Omega : v \in H \setminus \VnOm \} \cap \{ 
    \omega \in \Omega : e \in B \}\big)\\
    &=: (T_1 \cap T_2) \cup T_3.
  \end{align*}
  As the function $\omega \mapsto h_n \fn(\omega) + U^{n-1} - \big(I + h_n  
  \AnOm \big)v$ is measurable, it follows that $T_2 \subset \Omega$ is 
  measurable. The sets $T_1$ and $T_3$ are measurable by assumption. Thus, it 
  follows that $\omega \mapsto g(\omega, v)$ and therefore $\omega \mapsto 
  \dualV{g(\omega, v)}{w}$ is measurable.
  
  As argued above for every $\omega \in \Omega$, there exists a unique element 
  $U^n(\omega)$ such that $g(\omega, U^n(\omega)) = 0$.
  Thus, we can now apply Lemma~\ref{lem:measurability} to prove that $U^n \colon 
  \Omega \to H$ is $\F_n$-measurable.
\end{proof}

\begin{lemma}
  \label{lem:apriori}
  Let Assumptions~\ref{ass:spaces}--\ref{ass:expectation} be fulfilled. 
  Further, let the random variables $\fn \colon \Omega \to H$ be given such that they are 
  $\F_{\xi_n}$-measurable and $\E_{\xi_n} \big[ \|\fn\|_{H}^2 \big] < \infty$ for 
  every $n \in \{1,\dots, N\}$. 
  Then for $2\kappa h_n \leq 2\kappa h < 1$ the solution $\{U^n\}_{n \in \{1,\dots,N\}}$ 
  of \eqref{eq:scheme} fulfills the a priori bound
  \begin{align*}
    &\E_n \big[\|U^n\|_H^2\big] + \sum_{i=1}^{n}\E_i \big[ \|U^i - U^{i-1}\|_H^2 \big]
    + 2 \sum_{i=1}^{n} h_i \E_i \big[ \mui |U^i |_{\Vi}^2\big]\\
    &\leq C \Big( 2\|u_0\|^2  + 4 T \lambda+ 5 C T \sum_{i=1}^{N} h_i \E_{\xi_i} \big[ \| 
    \fii\|_{H}^2 \big]  \Big),
  \end{align*}
  where $C = \frac{1}{1- 2 h \kappa} \exp\big(\frac{2\kappa T}{1- 2 h \kappa}\big)$ for 
  all $n \in \{1,\dots,N\}$.
\end{lemma}

\begin{proof}
  We start by testing \eqref{eq:scheme} with the solution $U^i$ to find that
  \begin{align} \label{eq:proof_apriori_1}
    \inner{U^i - U^{i-1}}{U^i} + h_i \dualVi{\Ai U^i}{U^i}
    = h_i \inner{\fii}{U^i}.
  \end{align}
  For the first term of this equality, we use the identity $\inner{a - b}{a} = \frac{1}{2} 
  (\|a\|^2 - \|b\|^2 + \|a-b\|^2 )$ for $a, b \in H$ to find that
  \begin{align*}
    \inner{U^i - U^{i-1}}{U^i}
    = \frac{1}{2} \big( \|U^i\|_H^2 - \|U^{i-1}\|_H^2 + \|U^i - U^{i-1}\|_H^2 \big).
  \end{align*}
  Due to the coercivity condition from Assumption~\ref{ass:A} (v), we obtain
  \begin{align*}
    \dualVi{\Ai U^i}{U^i} + \kapi \|U^i \|_H^2 + \lami \geq \mui |U^i |_{\Vi}^p.
  \end{align*}
  For the right-hand side of \eqref{eq:proof_apriori_1}, we observe $\inner{\fii}{U^i} \leq 
  \| \fii\|_H \|U^i \|_H$.
  Combining the previous statements, we find
  \begin{align*}
    0&=\dualVi{U^i - U^{i-1} + h_i \Ai U^i - h_i \fii}{U^i}\\
    &\geq \frac{1}{2} \big( \|U^i\|_H^2 - \|U^{i-1}\|_H^2 + \|U^i - 
    U^{i-1}\|_H^2 \big)\\
    &\quad - h_i \kapi \|U^i \|_H^2 - h_i \lami + h_i \mui |U^i |_{\Vi}^p
    - h_i \| \fii\|_{H} \|U^i \|_H.
  \end{align*}
  After rearranging the terms and multiplying both sides of the inequality with the factor  
  $2$, we obtain the following bound 
  \begin{align*}
    &\|U^i\|_H^2 - \|U^{i-1}\|_H^2 + \|U^i - U^{i-1}\|_H^2 + 2 h_i \mui |U^i |_{\Vi}^p\\
    &\qquad\leq 2 h_i \kapi \|U^i \|_H^2 + 2 h_i \lami 
    + 2 h_i \| \fii\|_{H} \|U^i \|_H.
  \end{align*}
  Taking the expectation and using Assumption~\ref{ass:expectation} shows that
  \begin{align*}
    &\E_i \big[\|U^i\|_H^2\big] - \E_{i-1} \big[\|U^{i-1}\|_H^2\big] + \E_i \big[ \|U^i - 
    U^{i-1}\|_H^2\big] + 2 h_i \E_i \big[ \mui |U^i |_{\Vi}^p \big]\\
    &\qquad\leq 2 h_i \E_i \big[ \kapi \|U^i \|_H^2\big] + 2 h_i \E_{\xi_i} \big[\lami\big]
    +2 h_i \E_i \big[ \| \fii\|_{H} \|U^i \|_H\big] \\
    &\qquad\leq 2 h_i \kappa \E_i \big[ \|U^i \|_H^2\big] + 2 h_i \lambda
    + 2 h_i \E_i \big[ \| \fii\|_{H} \|U^i \|_H\big].
  \end{align*}
  This inequality is summed up from $i = 1$ to $n \in \{1,\dots,N\}$,
  \begin{align}\label{eq:proof_apriori_2}
    \begin{split}
      &\E_n \big[\|U^n\|_H^2\big] + 2 \sum_{i=1}^{n} \E_i \big[ \|U^i - U^{i-1}\|_H^2 \big]
      + 2 \sum_{i=1}^{n} h_i \E_i \big[ \mui |U^i |_{\Vi}^p\big]\\
      &\qquad\leq \|u_0\|_H^2 + 2 \kappa \sum_{i=1}^{n} h_i \E_i \big[ \|U^i \|_H^2\big] 
      + 2 T \lambda
      +2 \sum_{i=1}^{N} h_i \E_i \big[ \| \fii\|_{H} \|U^i \|_H\big],
    \end{split}
  \end{align}
  where we only made the right-hand side bigger by summing to the final value $N$.
  In the following, denote $i_{\max} \in \{1,\dots,N\}$ such that $\max_{i \in \{1,\dots,N\}} 
  \E_i \big[\|U^i \|_H^2 \big] = \E_{i_{\max}} \big[\|U^{i_{\max}} \|_H^2 \big]$. For the last term in~\eqref{eq:proof_apriori_2} we then have
  \begin{align*}
    2 \sum_{i=1}^{N} h_i \E_i \big[ \| \fii\|_{H} \|U^i \|_H\big]
    &\leq 2 \sum_{i=1}^{N} h_i \big(\E_i \big[ \| \fii\|_{H}^2 \big] 
    \big)^{\frac{1}{2}}
    \big(\E_i \big[\|U^i \|_H^2 \big]\big)^{\frac{1}{2}}\\
    &\leq 2 \big(\E_{i_{\max}} \big[\|U^{i_{\max}} \|_H^2 \big]\big)^{\frac{1}{2}} 
    \sum_{i=1}^{N} h_i \big(\E_{\xi_i} \big[ \| \fii\|_{H}^2 \big] \big)^{\frac{1}{2}}\\
    &\leq 2 \big(\E_{i_{\max}} \big[\|U^{i_{\max}} \|_H^2 \big]\big)^{\frac{1}{2}} 
    B, 
  \end{align*}
  where $B = \big( T \sum_{i=1}^{N} h_i \E_{\xi_i} \big[ \| \fii\|_{H}^2 
  \big] \big)^{\frac{1}{2}}$. We further abbreviate
   $x_n = \E_n \big[\|U^n\|_H^2\big] + \sum_{i=1}^{n} \E_i \big[ \|U^i - U^{i-1}\|_H^2 
   \big] + 2 \sum_{i=1}^{n} h_i \E_i \big[ \mui |U^i |_{\Vi}^p\big]$,
   and note that it follows directly from this definition that
  \begin{align*}
    2 \kappa \sum_{i=1}^{n} h_i \E_i \big[ \|U^i \|_H^2\big] \leq 2 \kappa \sum_{i=1}^{n} 
    h_i x_i.
  \end{align*}
  In conclusion, \eqref{eq:proof_apriori_2} therefore implies that
  \begin{align*}
    x_n &\leq \|u_0\|_H^2 + 2 \kappa \sum_{i=1}^{n} h_i x_i
    + 2 T \lambda 
    + 2 \big(\E_{i_{\max}} \big[\|U^{i_{\max}} \|_H^2 \big]\big)^{\frac{1}{2}} B.
  \end{align*}
  Applying the discrete Gr\"onwall inequality in Lemma~\ref{lem:gronwall}
  yields
  \begin{equation}\label{eq:proof_apriori_3}
      x_n \leq C \big(\|u_0\|_H^2  + 2 T \lambda 
      + 2 \big(\E_{i_{\max}} \big[\|U^{i_{\max}} \|_H^2 \big]\big)^{\frac{1}{2}} 
      B\big),
  \end{equation}
  for $C = \frac{1}{1- 2 h \kappa} \exp\big(\frac{2\kappa T}{1- 2 h \kappa}\big)$.  
  As this inequality holds for every $n \in \{1,\dots,N\}$, it is also fulfilled for $i_{\max}$. 
  Thus, it follows that
  \begin{align*}
    \E_{i_{\max}} \big[\|U^{i_{\max}}\|_H^2\big]
    \leq C \big(\|u_0\|^2 + 2 T \lambda + 2 \big(\E_{i_{\max}} \big[\|U^{i_{\max}} 
    \|_H^2 \big]\big)^{\frac{1}{2}} B \big).
  \end{align*}  
  We can now use that $x^2 \leq 2ax + b^2$ implies that $x \leq 2a +b$ for $a,b,x \in 
  [0,\infty)$ and find
  \begin{align*}
    \big(\E_{i_{\max}} \big[\|U^{i_{\max}}\|_H^2\big]\big)^{\frac{1}{2}}
    \leq C^{\frac{1}{2}}
    \big(\|u_0\|^2  + 2 T \lambda\big)^{\frac{1}{2}} + 2 C B.
  \end{align*}
  Inserting this bound in \eqref{eq:proof_apriori_3} and applying Young's inequality 
  (Lemma~\ref{lem:youngsInequality} with $\varepsilon = \frac{1}{2}$), we then obtain
  \begin{align*}
    &\E_n \big[\|U^n\|_H^2\big] + \sum_{i=1}^{n}\E_i \big[ \|U^i - U^{i-1}\|_H^2 \big]
    + 2 \sum_{i=1}^{n} h_i \E_i \big[ \mui |U^i |_{\Vi}^2\big]\\
    &\qquad\leq C \big(\|u_0\|^2  + 2 T \lambda
    + 2 C^{\frac{1}{2}} B
    \big(\|u_0\|^2  + 2 T \lambda\big)^{\frac{1}{2}}
    + 4 C B^2 \big)\\
    &\qquad\leq C \big(\|u_0\|^2  + 2 T \lambda
    + \big(\|u_0\|^2  + 2 T \lambda\big) + C B^2 + 4 C B^2 
    \big)\\
    &\qquad\leq C \big( 2\|u_0\|^2  + 4 T \lambda+ 5 C B^2 \big),
  \end{align*}  
  which finishes the proof.
\end{proof}

\section{Stability and convergence in expectation}
\label{sec:convH}

With the previous sections in mind, we can now turn our attention to the main results of 
this paper. We provide error bounds for the scheme \eqref{eq:scheme} measured in 
expectation. First, we give a stability result in Theorem~\ref{thm:stabilityH}. The 
stability bound can be proved in a similar manner to the a priori bound in 
Lemma~\ref{lem:apriori}. The aim 
of this bound is to show how two solutions of the same scheme with respect to different 
right-hand sides and initial values differ. 
This stability result can then be used to prove the desired error bounds in
Theorem~\ref{thm:convH} by using well-chosen data that agrees with the exact
solution at the grid points.
Note that in contrast to other works (e.g.~\cite{Emmrich.2009, 
EmmrichThalhammer.2010}), we measure $f(t) - A(t)u(t)$ in the 
$H$-norm. This can be interpreted as a stricter regularity assumption. The advantage 
is that certain error terms disappear in expectation, compare the second bound in 
Lemma~\ref{lem:auxillary_bounds_sumH}.

\begin{theorem}
  \label{thm:stabilityH}
  Let Assumptions~\ref{ass:spaces}--\ref{ass:expectation} be fulfilled.
  Further, let the random variable $\fn \colon \Omega \to H$ be given such that it is 
  $\F_{\xi_n}$-measurable and $\E_{\xi_n} \big[ \|\fn\|_H^2 \big] < \infty$ for 
  every $n \in \{1,\dots, N\}$. Let $\{U^n\}_{n \in \{1,\dots,N\}}$ be the solution of 
  \eqref{eq:scheme} and let $\{V^n\}_{n \in \{1,\dots,N\}}$ be the solution of 
  \begin{align} \label{eq:stabH}
    \begin{cases}
      V^n - V^{n-1} + h_n \An V^n = h_n g^n \quad &\text{in } V_{\xi_n}^*,
      \quad n \in \{1,\dots,N\}, \\
      V^0 = v_0 \quad &\text{in } H,
    \end{cases}    
  \end{align}
  for $v_0 \in H$ and $g^n \colon \Omega \to H$ such that it is 
  $\F_{\xi_n}$-measurable and $\E_{\xi_n} \big[ \|g^n\|_H^2 \big] < \infty$ for every $n 
  \in \{1,\dots, N\}$. 
  Then for $2\kappa h_n \leq 2\kappa h < 1$, it follows that 
  \begin{align*}
    &\E_n \big[ \|U^n - V^n\|_H^2\big] + 
    \frac{1}{2} \sum_{i=1}^{n} \E_i \big[\|U^i - V^i - (U^{i-1} - V^{i-1})\|_H^2 \big]\\
    &\quad+ 2 \sum_{i=1}^{n} h_i \E_i \big[ \etai |U^i - V^i |_{\Vi}^p\big] \\
    &\leq 2 C \|u_0 -v_0\|^2 + 4 C \sum_{i=1}^{N} h_i^2 \E_i \big[ \| \fii - g^i\|_H^2\big]
    + 5 C^2 T \sum_{i=1}^{N} h_i \big\| \E_{\xi_i} \big[ \fii - g^i \big] \big\|_H^2
  \end{align*}
  for $C = \frac{1}{1 - 2 h \kappa} \exp\big(\frac{2\kappa T}{1 - 2 \kappa T}\big)$ and  
  $n \in \{1,\dots,N\}$.
\end{theorem}

\begin{proof}
  We start by subtracting \eqref{eq:stabH} from \eqref{eq:scheme} and testing with $U^i - V^i$ 
  to get
  \begin{align} \label{eq:proof_stabH_1}
    \begin{split}
      &\innerb{(U^i - V^i) - (U^{i-1} - V^{i-1})}{U^i - V^i}\\
      &\quad+ h_n \dualVi{\Ai U^i - \Ai V^i}{U^i - V^i}
      = h_n \inner{\fii - g^i}{U^i - V^i}.
    \end{split}
  \end{align}
  For the first term of this equality, we use the identity $\inner{a - b}{a} = \frac{1}{2} (\|a\|^2 
  - \|b\|^2 + \|a-b\|^2 )$ for $a, b \in H$ to find that
  \begin{align*}
    &\innerb{(U^i - V^i) - (U^{i-1} - V^{i-1})}{U^i - V^i}\\
    &= \frac{1}{2} \big( \|U^i - V^i\|_H^2 - \|U^{i-1} - V^{i-1}\|_H^2 + \|U^i - V^i - (U^{i-1} - 
    V^{i-1})\|_H^2 \big).
  \end{align*}
  Due to the monotonicity condition from Assumption~\ref{ass:A} (iii), we obtain
  \begin{align*}    
    \dualVi{\Ai U^i - \Ai V^i}{U^i - V^i} 
    + \kapi \|U^i - V^i \|_H^2
    \geq \etai |U^i - V^i |_{\Vi}^p.
  \end{align*}
  It remains to find a bound for the right-hand side of \eqref{eq:proof_stabH_1}. 
  Applying Cauchy-Schwarz's inequality and 
  the weighted Young inequality for products (Lemma~\ref{lem:youngsInequality} with 
  $\varepsilon = 1$), it follows that
  \begin{align*}
    &h_i \innerb{\fii - g^i}{U^i - V^i}\\
    &= h_i \innerb{\fii - g^i}{U^{i-1} - V^{i-1}}
    + h_i \innerb{\fii - g^i}{U^i - V^i - (U^{i-1} - V^{i-1})}\\
    &\leq h_i \innerb{\fii - g^i}{U^{i-1} - V^{i-1}}
    + h_i \| \fii - g^i\|_H  \| U^i - V^i - (U^{i-1} - V^{i-1})\|_H\\
    &\leq h_i \innerb{\fii - g^i}{U^{i-1} - V^{i-1}}
    + h_i^2 \| \fii - g^i\|_H^2
    + \frac{1}{4} \| U^i - V^i - (U^{i-1} - V^{i-1})\|_H^2.
  \end{align*}
  Combining the previous statements, we find
  \begin{align*}
    0&= \inner{U^i - V^i - (U^{i-1} - V^{i-1})}{U^i - V^i}\\
    &\quad + h_i \dualVi{\Ai U^i - \Ai V^i }{U^i - V^i}
    - h_i \innerb{\fii - g^i}{U^i - V^i}\\
    &\geq \frac{1}{2} \big( \|U^i - V^i\|_H^2 - \|U^{i-1} - V^{i-1}\|_H^2 + \|U^i - V^i - 
    (U^{i-1} - V^{i-1})\|_H^2 \big)\\
    &\quad - h_i \kapi \|U^i - V^i \|_H^2 + h_i \etai |U^i - V^i |_{\Vi}^p\\
    &\quad -h_i \innerb{\fii - g^i}{U^{i-1} - V^{i-1}}
    - h_i^2 \| \fii - g^i\|_H^2
    - \frac{1}{4} \| U^i - V^i - (U^{i-1} - V^{i-1})\|_H^2.
  \end{align*}
  After rearranging the terms and multiplying both sides of the inequality with the factor  
  $2$, we obtain the following bound 
  \begin{align*}
    &\|U^i - V^i\|_H^2 - \|U^{i-1} - V^{i-1}\|_H^2 + \frac{1}{2}\|U^i - V^i - (U^{i-1} - 
    V^{i-1})\|_H^2\\
    &\qquad  + 2 h_i \etai |U^i - V^i |_{\Vi}^p\\
    &\quad\leq 2 h_i \kapi \|U^i - V^i \|_H^2
    + 2 h_i \innerb{\fii - g^i}{U^{i-1} - V^{i-1}}
    + 2 h_i^2 \| \fii - g^i\|_H^2.
  \end{align*}
  By first taking the $\E_{\xi_i}$-expectation of this inequality and then applying
  also the $\E_{i-1}$-expectation, we find that
  \begin{align*}
    &\E_i \big[ \|U^i - V^i\|_H^2\big] - \E_{i-1}\big[ \|U^{i-1} - V^{i-1}\|_H^2\big] + 
    \frac{1}{2} \E_i \big[\|U^i - V^i - (U^{i-1} - V^{i-1})\|_H^2 \big]\\
    &\qquad+ 2 h_i \E_i \big[ \etan |U^i - V^i |_{\Vi}^p\big] \\
    &\quad\leq  2 h_i \E_i \big[ \kapi \|U^i - V^i \|_H^2 \big]
    + 2 h_i \E_{i-1}\big[ \innerb{\E_{\xi_i} \big[\fii - g^i \big]}{U^{i-1} - 
      V^{i-1}}\big]\\
    &\qquad + 2 h_i^2 \E_{\xi_i} \big[ \| \fii - g^i\|_H^2\big].
  \end{align*}
  Combining the previous two inequalities and summing up from $i = 1$ to $n \in 
  \{1,\dots,N\}$, we obtain
  \begin{align}\label{eq:proof_stabH_2}
    \begin{split}
      &\E_n \big[\|U^n - V^n\|_H^2\big] + \frac{1}{2}\sum_{i=1}^{n} \E_i \big[ \|U^i - V^i - 
      (U^{i-1} - V^{i-1})\|_H^2 \big]\\
      &\qquad + 2 \sum_{i=1}^{n} h_i \E_i \big[ \etai |U^i - V^i |_{\Vi}^p\big]\\
      &\quad\leq \|u_0 - v_0\|_H^2 + 2 \kappa \sum_{i=1}^{n} h_i \E_i \big[ \|U^i - V^i \|_H^2\big] \\
      &\qquad + 2 \sum_{i=1}^{n} h_i \E_{i-1}\big[ \innerb{\E_{\xi_i} \big[\fii - g^i 
      \big]}{U^{i-1} - V^{i-1}}\big]
      + 2 \sum_{i=1}^{N} h_i^2 \E_{\xi_i} \big[ \| \fii - g^i\|_H^2\big],
    \end{split}
  \end{align}
  where we only made the right-hand side bigger by summing to the final value $N$.
  In the following, denote $i_{\max} \in \{1,\dots,N\}$ such that $\max_{i \in \{1,\dots,N\}} 
  \E_i \big[\|U^i - V^i \|_H^2 \big] = \E_{i_{\max}} \big[\|U^{i_{\max}} - V^{i_{\max}}\|_H^2 
  \big]$. By Lemma~\ref{lem:measurability}, it follows that $U^{i-1} - V^{i-1}$ is 
  $\F_{i-1}$-measurable and thus independent of the $\F_{\xi_i}$-measurable random 
  variable $\fii - g^i$. Therefore, we find that
  \begin{align*}
    &2 \sum_{i=1}^{n} h_i \E_{i-1}\big[ \innerb{\E_{\xi_i} \big[\fii - g^i \big]}{U^{i-1} 
    - V^{i-1}}\big]\\
    &\quad\leq 2 \sum_{i=1}^{n} h_i \big\| \E_{\xi_i} \big[ \fii - g^i \big] \big\|_H 
    \E_{i-1} \big[ \| U^{i-1} - V^{i-1}\|_H \big] \\
    &\quad\leq 2 \big(\E_{i_{\max}} \big[\|U^{i_{\max}} - V^{i_{\max}}\|_H^2 
    \big]\big)^{\frac{1}{2}} \sum_{i=1}^{N} h_i \big\| \E_{\xi_i} \big[ \fii - g^i \big] \big\|_H.
  \end{align*}
  To keep the presentation compact, we abbreviate 
  \begin{align*}
    B_1 = \sum_{i=1}^{N} h_i^2 \E_i \big[ \| \fii - g^i\|_H^2\big] \quad 
    \text{and} \quad
    B_2 = \sum_{i=1}^{N} h_i \big\| \E_{\xi_i} \big[ \fii - g^i \big] \big\|_H.
  \end{align*} 
  Setting 
  \begin{align*}
    x_n  &= \E_n \big[\|U^n - V^n\|_H^2\big] + \frac{1}{2}\sum_{i=1}^{n} \E_i \big[ \|U^i - V^i - 
    (U^{i-1} - V^{i-1})\|_H^2 \big]\\
    &\quad + 2 \sum_{i=1}^{n} h_i \E_i \big[ \etai |U^i - V^i |_{\Vi}^p\big],
  \end{align*}
  we have $2\kappa \sum_{i=1}^{n}{ h_i \E_i \big[ \|U^i - V^i \|_{H}^2\big]} \le 2\kappa 
  \sum_{i=1}^{n}{ h_i x_i}$.
  We can now apply Gr\"onwall's inequality (Lemma~\ref{lem:gronwall}) to
  \eqref{eq:proof_stabH_2}. It follows that
  \begin{equation}\label{eq:proof_stabH_3}
      x_n \leq C
      \Big(\|u_0 - v_0\|^2  + 2 B_1
      + 2 \big(\E_{i_{\max}} \big[\|U^{i_{\max}} - V^{i_{\max}} \|_H^2 
      \big]\big)^{\frac{1}{2}} B_2 \Big),
  \end{equation}
  for $C = \frac{1}{1- 2 h \kappa} \exp\big(\frac{2\kappa T}{1- 2 h \kappa}\big)$.  
  As this inequality holds for every $n \in \{1,\dots,N\}$, it is also fulfilled for $i_{\max}$. 
  Thus, it follows that
  \begin{align*}
    &\E_{i_{\max}} \big[\|U^{i_{\max}} - V^{i_{\max}}\|_H^2\big]\\
    &\quad\leq C \big(\|u_0 -v_0\|^2 + 2 B_1+ 2 \big(\E_{i_{\max}} \big[\|U^{i_{\max}} 
    - V^{i_{\max}}\|_H^2 \big]\big)^{\frac{1}{2}} B_2\big).
  \end{align*}  
  We can now use that $x^2 \leq 2ax + b^2$ implies that $x \leq 2a +b$ for $a,b,x \in 
  [0,\infty)$ and find
  \begin{align*}
    \big(\E_{i_{\max}} \big[\|U^{i_{\max}}- V^{i_{\max}}\|_H^2\big]\big)^{\frac{1}{2}}
    \leq C^{\frac{1}{2}}
    \big(\|u_0 -v_0\|^2 + 2 B_1\big)^{\frac{1}{2}} + 2 C B_2.
  \end{align*}
  Inserting this bound in \eqref{eq:proof_stabH_3} and applying Young's inequality 
  (Lemma~\ref{lem:youngsInequality} for $\varepsilon = 1$), we then obtain
  \begin{align*}
    &\E_n \big[\|U^n -V^n\|_H^2\big] + \frac{1}{2}\sum_{i=1}^{n}\E_i \big[ \|U^i - V^i - 
    (U^{i-1} - V^{i-1})\|_H^2 \big]\\
    &\qquad + 2 \sum_{i=1}^{n} h_i \E_i \big[ \etai |U^i - V^i |_{\Vi}^2\big]\\
    &\quad\leq C \Big(\|u_0 -v_0\|^2 + 2 B_1
    + 2 C^{\frac{1}{2}}
    \big(\|u_0 - v_0\|^2  + 2 B_1\big)^{\frac{1}{2}} B_2+ 4 C B_2^2 \Big)\\
    &\quad\leq C \Big(\|u_0 -v_0\|^2 + 2 B_1
    + \big(\|u_0 - v_0\|^2  + 2 B_1\big) + C B_2^2 + 4 C B_2^2 \Big)\\
    &\quad= 2 C \big(\|u_0 -v_0\|^2 + 2 B_1\big) + 5 C^2 B_2^2.
  \end{align*}  
  It only remains to insert
  \begin{align*}
    B_2^2 
    = \Big( \sum_{i=1}^{N} h_i \big\| \E_{\xi_i} \big[ \fii - g^i \big] \big\|_H \Big)^2
    \leq T \sum_{i=1}^{N} h_i \big\| \E_{\xi_i} \big[ \fii - g^i \big] \big\|_H^2,
  \end{align*}
  to finish the proof.
\end{proof}

\begin{theorem}
  \label{thm:convH}
  Let Assumptions~\ref{ass:spaces}--\ref{ass:expectation} be fulfilled. 
  Further, let $f_{\xi_n} \in C([0,T]; H)$ almost surely and let $\fn = f_{\xi_n}(t_n) \in 
  L^2(\Omega;H)$ for all $n \in \{1,\dots,N\}$.
  Let $\{U^n\}_{n \in \{1,\dots,N\}}$ be the solution of \eqref{eq:scheme} and $u$ be the 
  solution of \eqref{eq:equation} that fulfills $u' \in C^{\gamma} ([0,T]; H)$, $\gamma \in 
  (0,1]$. 
  Moreover, let $\An u(t_n) \in L^2(\Omega; H)$ be fulfilled.
  
  Then for $2\kappa h_n \leq 2\kappa h < 1$ and $e^n = U^n - u(t_n)$, it follows that
  \begin{align*}
    &\E_n \big[\|e^n\|_H^2\big] + \frac{1}{2}\sum_{i=1}^{n}\E_i \big[ \|e^i - e^{i-1}\|_H^2 \big]
    + 2 \sum_{i=1}^{n} h_i \E_i \big[ \etai |e^i |_{\Vi}^p\big]\\
    &\quad\leq 8 h C \sum_{i = 1}^{N} h_i \E_{\xi_i} \big[ \big\| f_{\xi_i}(t_i) - \Ai u(t_i)
    - (f(t_i) - A(t_i) u(t_i))  \big\|_H^2 \big] \\
    &\qquad + 4 h^{1 + 2 \gamma} C |u|_{C^{\gamma}([0,T];H)}^2 T
    + 5 h^{2\gamma} C^2 |u'|_{C^{\gamma}([0,T];H)} T^2,
  \end{align*}
  where $C = \frac{1}{1- 2 h \kappa} \exp\big(\frac{2\kappa 
  T}{1- 2 h \kappa}\big)$ for all $n \in \{1,\dots,N\}$.
\end{theorem}

\begin{proof}
  We use $\{V^n\}_{n \in \{1,\dots,N\}}$ given by
  \begin{align*}
    \begin{cases}
      V^n - V^{n-1} + h_n \An V^n = h_n g^n \quad &\text{in } V_{\xi_n}^*, \quad n \in 
      \{1,\dots,N\}, \\
      V^0 = u_0 \quad &\text{in } H,
    \end{cases}
  \end{align*}
  where
  \begin{align*}
    g^n = \frac{1}{h_n} \big( u(t_n) - u(t_{n-1})\big) + \An u(t_n) \in L^2(\Omega; H).
  \end{align*}
  With this particular choice of $g^n$, we can now show that $V^n = u(t_n)$ for every 
  $n \in \{1.\dots,N\}$.
  Given the initial value $u_0$,  the solution $V^1$ is then given by
  \begin{align*}
    V^1 &= u_0 + h_1 g^1 - h_1 A_{\xi_1}(t_1) V^1\\
    &= u_0 + \big( u(t_1) - u(t_{0})\big)
    + h_1 A_{\xi_1}(t_1) u(t_1) - h_1 A_{\xi_1}(t_1) V^1\\
    &= u(t_1) + h_1 A_{\xi_1}(t_1) u(t_1) - h_1 A_{\xi_1}(t_1) V^1.
  \end{align*}
  Therefore, it follows that
  \begin{align*}
    (I + h_1 A_{\xi_1}(t_1) ) V^1 = (I + h_1 A_{\xi_1}(t_1) ) u(t_1) \quad \text{in } 
    V_{\xi_1}^*.
  \end{align*}
  Since $I + h_1 A_{\xi_1}(t_1)$ is injective, we find $V^1 = u(t_1)$ in $V_{\xi_1}$. 
  Recursively, it follows that $V^n = u(t_n)$ in $\Vn$ for all other $n \in \{1,\dots, N\}$.
  Together with the stability estimate from Theorem~\ref{thm:stabilityH} we find for 
  $e^n = U^n - V^n = U^n - u(t_n)$
  that
  \begin{align*}
    \E_n \big[\|e^n\|_H^2\big] + \frac{1}{2}\sum_{i=1}^{n}\E_i \big[ \|e^i - e^{i-1}\|_H^2 \big]
    + 2 \sum_{i=1}^{n} h_i \E_i \big[ \etai |e^i |_{\Vi}^2\big]
    \leq 4 C B_1 + 5 C^2 B_2^2,
  \end{align*}
  where 
  \begin{align*}
    B_1 &= \sum_{i=1}^{N} h_i^2 \E_i \big[ \| \fii - g^i\|_H^2\big], \quad 
    B_2 = \sum_{i=1}^{N} h_i \big\| \E_{\xi_i} \big[ \fii - g^i \big] \big\|_H,\\
    C &= \frac{1}{1 - 2 h \kappa} \exp\Big(\frac{2\kappa T}{1 - 2 \kappa T}\Big).
  \end{align*} 
  Applying  Lemma~\ref{lem:auxillary_bounds_sumH} for $u' \in 
  C^{\gamma}([0,T];H)$, it follows that
  \begin{align*}
    B_1&\leq h \sum_{i=1}^{N} h_i \E_{\xi_i} \Big[ \Big\| f_{\xi_i}(t_i) - \Ai u(t_i) 
    - \frac{1}{h_i} \int_{t_{i-1}}^{t_i} \big(f(t) - A(t) u(t)\big) \diff{t} 
    \Big\|_H^2 \Big] \\
    &\leq 2 h \sum_{i = 1}^{N} h_i \E_{\xi_i} \big[ \big\| f_{\xi_i}(t_i) - \Ai u(t_i)
    - (f(t_i) - A(t_i) u(t_i))  \big\|_H^2 \big] \\
    &\qquad + 2 h^{1 + 2 \gamma} |u|_{C^{\gamma}([0,T];H)}^2 T 
  \end{align*}
  and
  \begin{align*}
    B_2^2 
    &\leq T \sum_{i=1}^{N} h_i \Big\| \E_{\xi_i} \Big[ f_{\xi_i}(t_i) - \Ai u(t_i) 
    - \frac{1}{h_i} \int_{t_{i-1}}^{t_i} \big(f(t) - A(t) u(t)\big) \diff{t}  \Big] \Big\|_H^2\\
    &\leq h^{2\gamma} |u'|_{C^{\gamma}([0,T];H)} T^2.
  \end{align*}
  Altogether, we obtain
  \begin{align*}
    &\E_n \big[\|e^n\|_H^2\big] + \frac{1}{2}\sum_{i=1}^{n}\E_i \big[ \|e^i - e^{i-1}\|_H^2 
    \big] + 2 \sum_{i=1}^{n} h_i \E_i \big[ \etai |e^i |_{\Vi}^2\big]\\
    &\quad\leq 8 h C \sum_{i = 1}^{N} h_i \E_{\xi_i} \big[ \big\| f_{\xi_i}(t_i) - \Ai u(t_i)
    - (f(t_i) - A(t_i) u(t_i))  \big\|_H^2 \big] \\
    &\qquad + 4 h^{1 + 2 \gamma} C |u|_{C^{\gamma}([0,T];H)}^2 T
    + 5 h^{2\gamma} C^2 |u'|_{C^{\gamma}([0,T];H)} T^2.
  \end{align*}
\end{proof}

\begin{remark}
  The main results can all be modified to a slightly different setting, where the right-hand side 
  $f(t)$ takes values in $V^*$ and where the family $\{\xi_n\}_{n \in 
  \N}$ of random variables does not have to be mutually independent. In 
  return, this setting requires slightly stronger assumptions on the operator $A(t)$. 
  First, we assume additionally that there exists a constant $c_V \in (0,\infty)$ such that 
  $\|\cdot\|_V \leq c_V \big( \|\cdot \|_H + |\cdot|_V\big)$ is fulfilled.
  To generalize the a priori bound from Lemma~\ref{lem:apriori} and the stability 
  results from 
  Theorem~\ref{thm:stabilityH}, we need to assume that $\mu_A$ from 
  Assumption~\ref{ass:A}~(v) and $\eta_A$ from Assumption~\ref{ass:A}~(iii) are 
  strictly positive, respectively. Moreover, if there exist $\gamma \in (0,1]$ and 
  $C \in [0,\infty)$ such that
  \begin{align*}
    \sum_{i = 1}^{N} h_i \E_i \big[ \big\| f_{\xi_i}(t_i) - \Ai u(t_i) - (f(t_i) - A(t_i) u(t_i)) 
    \big\|_{\Vi^*}^2\big] \leq C h^{2 \gamma}
  \end{align*}
  is fulfilled and $u' \in C^{\gamma}([0,T];H)$, we obtain similar error bounds.
  We omit the proofs, which are very similar to the ones presented above.
\end{remark}

\section{Numerical experiments}
\label{sec:numExperiments}

To illustrate the theoretical convergence results for the randomized scheme in practice, we apply it to the parabolic differential equation \eqref{eq:parabolicPDE} as discussed in Section~\ref{sec:example}.
This boundary, initial-value problem fits our setting as already explained there. We 
also consider what happens when we replace the nonlinear diffusion term with linear 
diffusion, and a smoother exact solution.

In both cases, we consider the problem on the spatial domain $\D = [-1,1]
\times [-1,1]$ which we split into rectangular sub-domains $\D_{\ell}$, $\ell \in 
\{1,\ldots, s\}$, with $M_x$ rectangles along the $x$-axis and $M_y$ rectangles along 
the 
$y$-axis. We choose $\D_{\ell}$ such that they have an overlap of $0.2$ on all internal 
sides. This means that with $M_x = M_y = 3$, we have $s = M_x M_y = 9$ 
sub-domains with, e.g., $\D_{1} = [-1, -0.267] \times [-1, -0.267]$, $\D_{2} = [-0.467, 
0.467] \times [-1, -0.267]$ and $\D_{5} = [-0.467, 0.467] \times [-0.467, 0.467]$. Note 
that they are not uniform in size, because the sub-domains adjacent to the outer edge 
of $\D$ have no overlap on one or two sides.

We have to choose a strategy for which sub-problems to select in each time step, i.e.\ 
specify the probabilities $\P(\Omega_{\xi = B})$ for $B \subset 2^{\{1,\dots,s\}}$. We 
consider two strategies. In the first, 
we simply use $\P(\Omega_{\xi = \{\ell\}}) = 1/s$. Thus every sub-domain is equally 
probable to be chosen. As a minor variation, we instead select a set of $k$ 
sub-domains by drawing with replacement according to the uniform probabilities. 

In the second strategy, we make use of a predictor. In addition to the stochastic 
approximation, we compute a deterministic approximation $Z^n$ using the backward 
Euler method, but on a coarser spatial mesh. The idea is that while this approximation 
is less accurate, it should be significantly cheaper to compute and still resemble the 
true solution. In the $n^{\text{th}}$ time step, we compute $\Psi_n = |Z^{n-1}| + |Z^n| + 
|\tilde{f}(t_n, \cdot)| > 10^{-3}$. This function is either $0$ or $1$ and indicates where 
in the domain something is actually happening. For each sub-domain, we then check 
whether it is ``sufficiently active'' or not by evaluating $\|\Psi_n \chi_l\| \ge \rho 
\|\Psi_n\|$ for a parameter $\rho \in (0,1)$. We select the set of those sub-domains 
which 
pass the test with probability $1-\rho$ and the set of all the other sub-domains with 
probability $\rho$.

\subsection{A nonlinear example} \label{sec:nonlin_exp}
In our first experiment, we use the problem parameters $T = 1$, $p = 4$ and 
$\alpha(t) \equiv 1$. Further, we choose the source term $\tilde{f}$ such that the exact 
solution is given by $u(t, x, y) = \tilde{u}(x - r \cos(2\pi t), y - r \sin(2\pi t))$ with $r = 
1/2$,
\begin{equation*}
  \tilde{u}( x, y) = \Bigl[0.03 - \frac{10^{3/8}}{4} (x^2 + y^2)^{\frac{4}{3}} 
  \Bigr]_+^{\frac{3}{4}}
\end{equation*}
and $[\cdot]_+ = \max\{\cdot, 0\}$. This describes a localized pulse that starts centered 
at $(0.5, 0)$ and which then rotates around the origin at the constant distance $r$. 
The shape of the pulse is inspired by the closed-form Barenblatt solution to $\partial_t 
u = \nabla \cdot (|\nabla u(t,x)|^{p-2}\nabla u)$, see e.g.~\cite{KaminVazquez.1988}. At 
$t = 0$, this solution is a Dirac delta, which then expands into a cone-shaped peak for 
$t>0$. Our pulse is this solution frozen at the time $t = 0.001$. We note that
due to the sharp interface where the pulse meets the $x$-$y$-plane and to the sharp 
peak, $u$ is of low regularity.

We discretize the problem in space using central finite differences, such that the 
approximation of the $p$-Laplacian is 2nd-order accurate. We use $41$ computational 
nodes in each spatial dimension, for a total of $1681$ degrees of freedom. For the 
temporal discretization, we use the scheme~\eqref{eq:scheme}, along with one of the 
two strategies outlined above. For the first strategy, we try $k = 1$ and $k = 2$. For 
the second, we evaluate the different parameters $\rho = 0.01, 0.05, 0.1, 0.2$. 
We compute approximations for the different (constant) time steps $h_n = 2^{-5}, 
2^{-6}, \ldots, 2^{-13}$ and estimate their corresponding errors at the final time by running the method with $50$ random iterations and averaging. That is, we approximate
\begin{equation*}
  \E_N \big[\|e^N\|_H^2\big] \approx \frac{1}{50} \sum_{j=1}^{50}{ \|U^N - U_{\text{ref}}\|_H^2},
\end{equation*}
where $U_{\text{ref}}$ is the exact solution $u(t_N, \cdot, \cdot)$ evaluated at the spatial grid.

Figure~\ref{fig:nonlinear} shows the resulting relative errors vs.\ the time steps, with the first 
strategy in the left plot and the second strategy in the right. We observe that both 
strategies result in errors that decrease as $\mathcal{O}(h^{1/2})$, in line with 
Theorem~\ref{thm:convH}. We note, however, that the errors for the first strategy are 
noticeably larger than those of the second strategy. We have also used fewer 
sub-domains for the first strategy, $M_x = 3$ and $M_y = 1$, rather than the $M_x = 
3$ and $M_y = 3$ which we used for the second strategy. This is because the 
nonlinear $p$-Laplacian provides less smoothing than its linear counterpart, the 
Laplacian. The error incurred by choosing the ``wrong'' sub-domain therefore decays 
slowly, which makes this strategy work less well with too many sub-domains. The 
second strategy works better with more sub-domains, since it essentially adaptively 
groups them into only two larger sub-domains; the active set and the inactive set. 
Increasing the number of sub-domains increases the fidelity such that the choice of 
whether each sub-domain is active or not becomes easier, albeit at a higher 
computational cost. If the spatial discretization is using finite elements, the
limit case would be when every element is its own subdomain. This is what is
considered in~\cite{StoneGeigerLord.2017} for a deterministic scheme, where it is,
indeed, observed that the overhead costs can be prohibitive even when using very
efficient data structures.

\begin{figure}
  \includegraphics[width=0.49\textwidth]{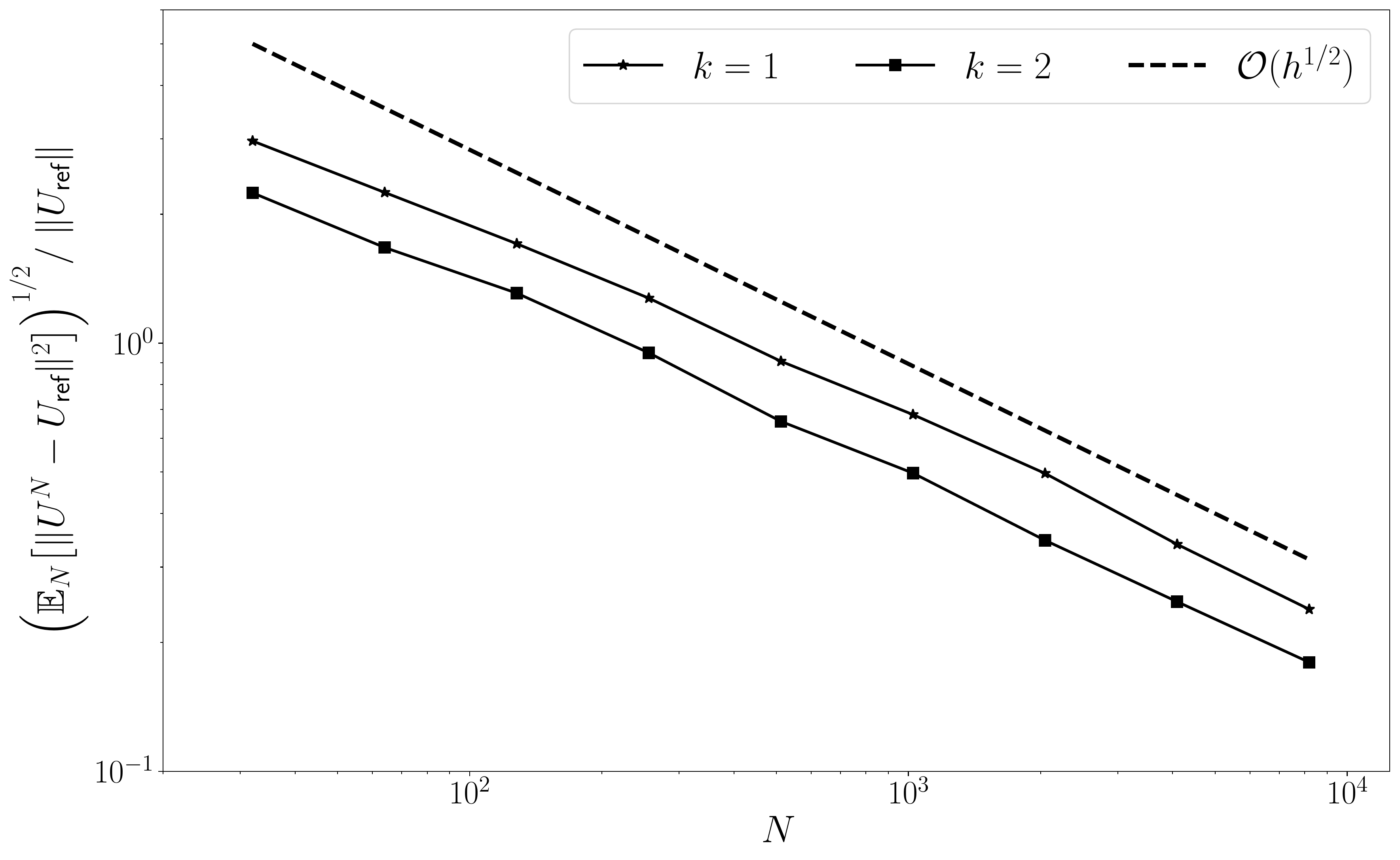}
  \includegraphics[width=0.49\textwidth]{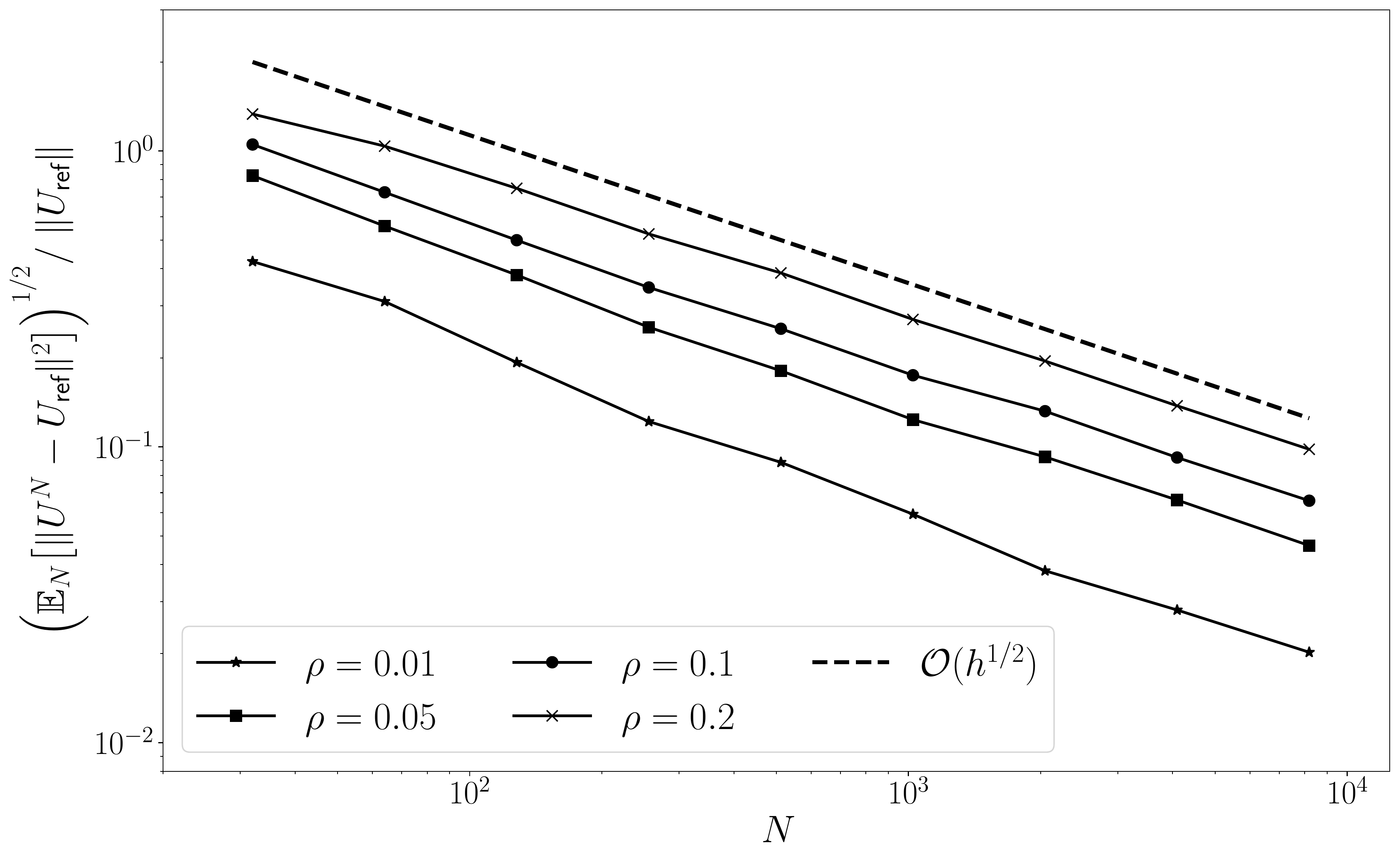}
  \caption{The relative errors $\big(\E_N\big[ \|U^N - U_{\text{ref}}\|^2 \big] \big)^{1/2} / \; \|U_{\text{ref}}\|$ for the nonlinear setting described in Section~\ref{sec:nonlin_exp}. The left plot uses the first randomized strategy and the right plot uses the second strategy. We observe that the errors decay as $\mathcal{O}(h^{1/2})$, in line with Theorem~\ref{thm:convH}, irrespective of the choice of $\rho$ or $k$. A smaller $\rho$ or larger $k$ decreases the error, but of course also incurs a higher computational cost.
    }
  \label{fig:nonlinear}
\end{figure}

\subsection{A linear example} \label{sec:lin_exp}
As a second experiment, we consider a linear version of the previous problem. We use the same parameters as in the previous section, except that we set $p = 2$ and $\alpha(t) = 0.1$, and that the rotating pulse is now Gaussian rather than a sharp peak. More precisely, the exact solution is given by
\begin{equation*}
 u(t, x, y) =  e^{-100 (x - r \cos(2\pi t))^2 - 100 (y -  r \sin(2\pi t))^2}.
\end{equation*}

The resulting errors are shown in Figure~\ref{fig:linear}. Again, we note that the first, 
uniform, strategy converges as $\mathcal{O}(h^{1/2})$, in line with 
Theorem~\ref{thm:convH}. The second strategy with $\rho = 0.01$ performs 
significantly better and converges as $\mathcal{O}(h)$ until the spatial error starts to 
dominate. This is essentially the same behaviour as if we would apply backward Euler 
to 
the full problem, but the method only updates the approximation on the most relevant 
sub-domains and is therefore cheaper to evaluate. This improved convergence order 
is possible due to the extra smoothness present in this linear problem. In the error 
bound of Theorem~\ref{thm:convH}, the first term becomes small due to the used 
strategy, and because the solution is smooth the remaining terms are of size $h^3$ 
and $h^2$, respectively.

Increasing the parameter $\rho$ means that we disregard more of the information from 
the predictor, and as seen in Figure~\ref{fig:linear} this causes the convergence order 
to decrease towards $1/2$. On the other hand, setting $\rho = 0$ means that we 
always choose all the sub-domains and thereby do more computations than if we 
would simply solve the full problem directly. The parameter $\rho$ is therefore a 
design parameter, and further research is required on how to choose it optimally for 
specific problem classes. Regardless of the choice, however, we still have 
$\mathcal{O}(h^{1/2})$-convergence.
\begin{figure}
  \includegraphics[width=0.49\textwidth]{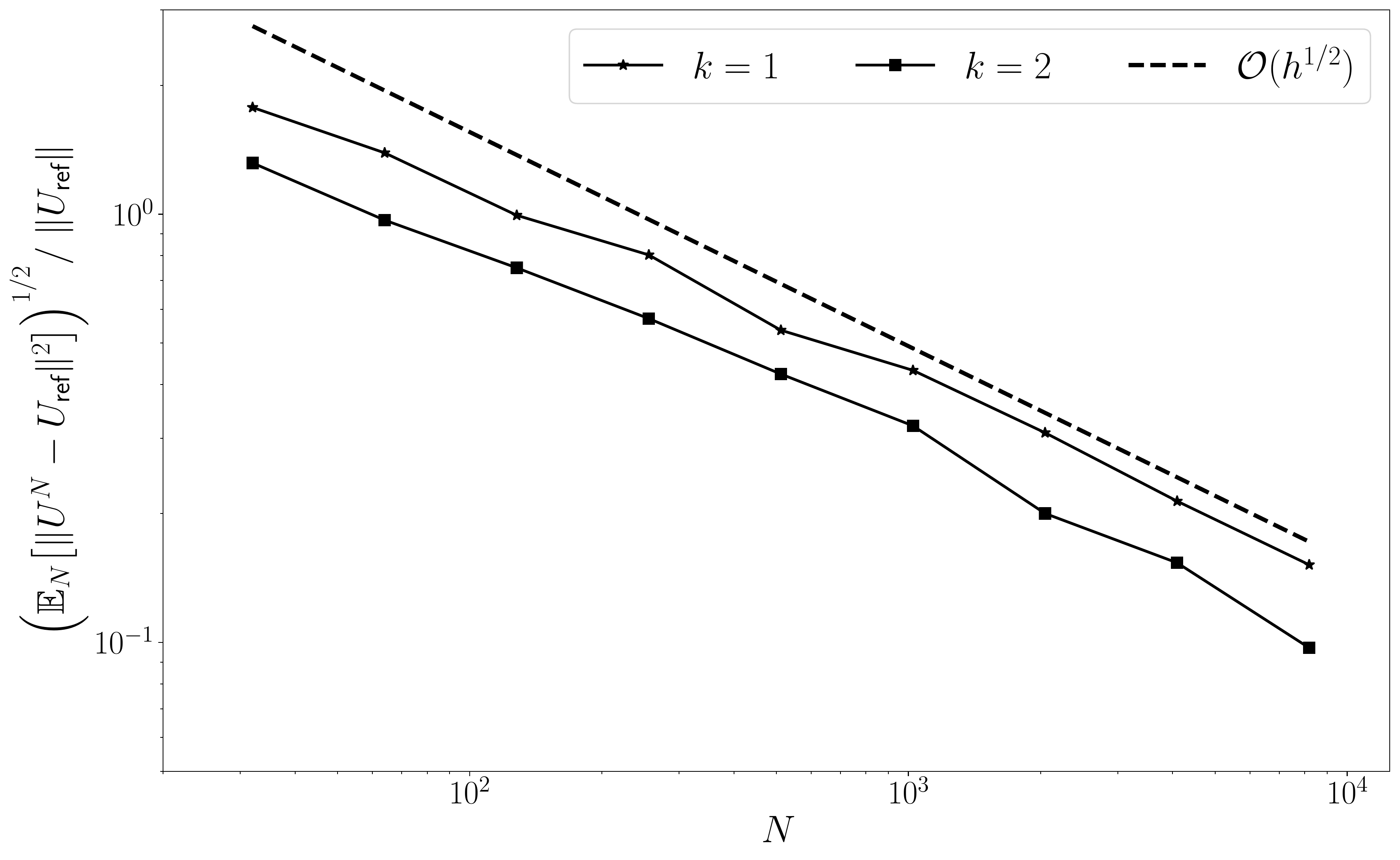}
  \includegraphics[width=0.49\textwidth]{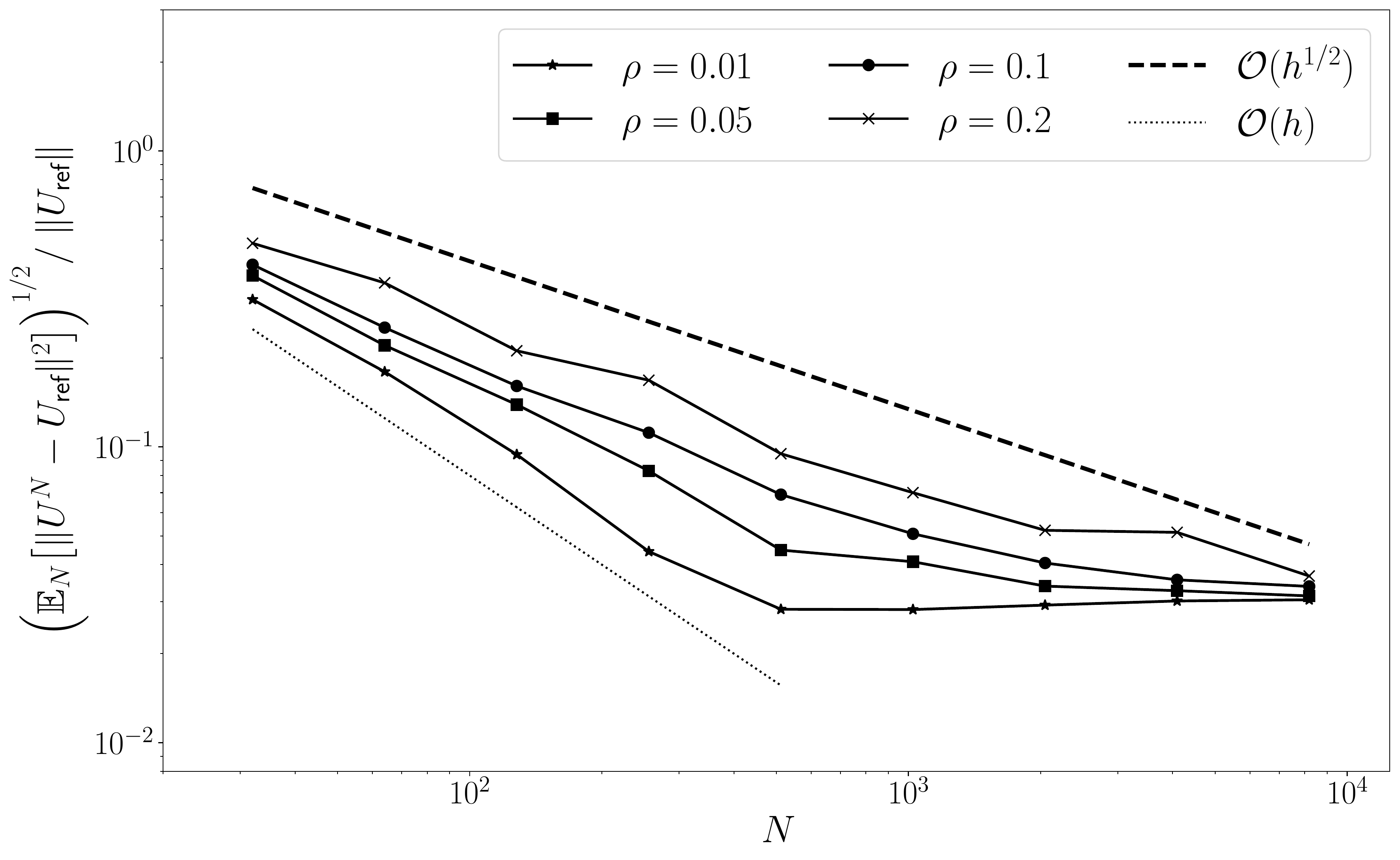}
  \caption{The relative errors $\big(\E_N\big[ \|U^N - U_{\text{ref}}\|^2 \big] \big)^{1/2} / \; 
  \|U_{\text{ref}}\|$ for the linear setting described in Section~\ref{sec:lin_exp}. The left 
  plot uses the first randomized strategy and the right plot uses the second strategy. 
  We observe that the errors for the first strategy decay as $\mathcal{O}(h^{1/2})$, 
  similarly to the nonlinear case. For the second strategy, large $\rho$ also leads to 
  convergence of order 1/2, while sufficiently small $\rho$ leads to faster convergence 
  of order $1$. We note that the errors plateau at around $2 \cdot 10^{-2}$ because we 
  compare the numerical approximations to the exact solution; this is the size of the 
  spatial error.    }
  \label{fig:linear}
\end{figure}

\appendix  
\section{Auxiliary results}
\label{sec:appendix}
In this appendix, we collect a few useful inequalities and technical results that are 
needed in the paper.

\begin{lemma}[Gr\"onwall]\label{lem:gronwall}
  Let $(u_n)_{n \in N}$ and $(b_n)_{n \in N}$ be two nonnegative sequences that 
  satisfy, for given $a \in [0,\infty)$ and $n \in \N$, that $u_n \leq a + \sum_{i=1}^{n} b_i 
  u_i$. For $b_n \in [0,1)$, it then follows that 
    \begin{align*}
      u_n \leq \frac{a}{1 - b_n} \exp\Big(\sum_{i=1}^{n-1} \frac{b_i}{1 - b_n} \Big).
    \end{align*}
\end{lemma}
  
\begin{lemma}[Scaled Young's inequality]
  \label{lem:youngsInequality}
  For $a, b \in [0,\infty)$, $\varepsilon \in (0,\infty)$, it follows that $a b \leq \varepsilon 
  a^2 + \frac{1}{4 \varepsilon} b^2$.
\end{lemma}

A proof can be found in \cite[Appendix~B.2~d]{Evans.1998}.

\begin{lemma}
  \label{lem:measurability}
  Let Assumptions~\ref{ass:spaces}--\ref{ass:expectation} be fulfilled. 
  Let $Q \subseteq V$ be a countable, dense subset of $V$, $V_{\xi}$ and $H$.
  Let the function $g \colon \Omega \times H \to V^*$ be given. Further, for $v \in H$ 
  the mapping $\omega \mapsto \dualV{g(\omega, v)}{w}$ is measurable for $v \in H$ 
  and $w \in Q$ and for almost every $\omega \in \Omega$ the mapping $v \mapsto 
  \dualV{g(\omega, v)}{w}$ continuous for every $v, w \in V_{\xi(\omega)}$.
  For every $\omega \in \Omega$, the function $g$ has a unique root which lies in 
  $V_{\xi(\omega)}$. 
  We denote this root by $r(\omega) \in V_{\xi(\omega)} $, i.e. $g(\omega, r(\omega)) = 
  0$.
  Then the function $r \colon \Omega \to H$ is measurable.
\end{lemma}

A similar proof can be found in \cite[Lemma~2.1.4]{EisenmannPhD.2019} and 
\cite[Lemma~4.3]{EKKL.2019}. The main difference in this version is that the function 
$g$ maps from $\Omega \times H$ instead of $\Omega \times V$ and therefore some 
small technical alterations have to be considered.

\begin{proof}[Proof of Lemma~\ref{lem:measurability}]
  To prove that $r$ is measurable, we show that $r^{-1}(B) \in \F$ for every open set $B$ 
  in $H$. First, we notice that
  \begin{align*}
    r^{-1} (B) 
    &= \{ \omega \in \Omega \colon r(\omega) \in B \} \\
    &= \{ \omega \in \Omega \colon \text{ there exists $u \in B$ such that } g(\omega, u) = 
    0 \} \\
    &= \{ \omega \in \Omega \colon \text{ there exists $u \in B$ such that } 
    \dualV{g(\omega, u)}{v} = 0\\ 
    &\hspace{7cm}\text{ for all } v \in Q, \|v\|_V = 1\} \\
    &= \bigcap_{ v \in Q, \|v\|_V=1} \{ \omega \in \Omega \colon \text{ there exists $u 
      \in B$ such that } \dualV{g(\omega, u)}{v} = 0 \} \\
    &= \bigcap_{ v \in Q, \|v\|_V=1} \bigcup_{ u \in B} \{ \omega \in \Omega \colon 
    \dualV{g(\omega, u)}{v} = 0 \}.
  \end{align*}
  Since $\omega \mapsto \dualV{g(\omega, v)}{w}$ is measurable for $v \in H$ 
  and $w \in Q$, the set 
  \begin{align*}
    \{ \omega \in \Omega \colon \dualV{g(\omega, u)}{v} = 0 \}
    = \big(\dualV{g(\cdot, u)}{v}\big)^{-1}(0)
  \end{align*} 
  is an element of $\F_{\xi}$ for $v \in Q$ and $u \in H$. If the set $B$ only contains a 
  countable amount of elements, it follows directly that $r^{-1} (B) \in \F_{\xi}$. 
  
  In the following, it remains to address the cases where $B$ is not countable.  
  For $\varepsilon \in (0,\infty)$ small enough and a fixed $v \in Q$, we introduce the 
  multi-valued mapping
  \begin{align*}
    r_{\varepsilon}^v \colon \Omega \to 2^H, \quad
    r_{\varepsilon}^v (\omega) &= \{  u \in H \colon | \dualV{g(\omega, u)}{v} | < \varepsilon 
    \}.
  \end{align*}
  For $B \subseteq H$ open, it follows that
  \begin{align*}
    \big(r_{\varepsilon}^v\big)^{-1}(B) 
    &= \{ \omega \in \Omega \colon r_{\varepsilon}^v(\omega) \in B \}\\
    &= \{ \omega \in \Omega \colon \text{ there exists $u \in B$ such that } | 
    \dualV{g(\omega, u)}{v} | 
    < \varepsilon \}\\
    &= \bigcup_{ u \in B} \{ \omega \in \Omega \colon | \dualV{g(\omega, u)}{v} | < 
    \varepsilon \}.
  \end{align*}
  In the following, we will show that
  \begin{align*}
    \big(r_{\varepsilon}^v\big)^{-1}(B) = \big(r_{\varepsilon}^v\big)^{-1}(B \cap Q). 
  \end{align*}
  Since $B \cap Q \subseteq B$, it directly follows that $\big(r_{\varepsilon}^v\big)^{-1}(B 
  \cap Q) \subseteq \big(r_{\varepsilon}^v\big)^{-1}(B)$. 
  It remains to verify that $\big(r_{\varepsilon}^v\big)^{-1}(B) \subseteq 
  \big(r_{\varepsilon}^v\big)^{-1}(B \cap Q)$. Let $\omega \in 
  \big(r_{\varepsilon}^v\big)^{-1}(B)$, i.e. there exists $u \in B$ such that
  \begin{align*}
    u \in r_{\varepsilon}^v(\omega) = \{ w \in H \colon | \dualV{g(\omega, w)}{v}| < 
    \varepsilon \}.
  \end{align*}
  Since $v \mapsto \dualV{g(\omega, v)}{w}$ is continuous for every $v, w \in 
  V_{\xi(\omega)}$ and $Q$ is dense in $H$, there exists 
  $u_Q \in B \cap Q$ such 
  that $| \dualV{g(\omega, u_Q)}{v}| < \varepsilon$. Thus, $u_Q \in 
  r_{\varepsilon}^v(\omega)$ and in particular 
  $\omega \in \big(r_{\varepsilon}^v\big)^{-1}(B \cap Q)$.
  This shows altogether that $\big(r_{\varepsilon}^v\big)^{-1}(B) = 
  \big(r_{\varepsilon}^v\big)^{-1}(B \cap Q)$.
  
  We can now finish the proof as
  \begin{align*}
    r^{-1}(B) 
    = \bigcap_{v \in Q, \|v\|_V = 1} \bigcap_{i \in \N} \big(r_{\frac{1}{i}}^v\big)^{-1}(B) 
    = \bigcap_{v \in Q, \|v\|_V = 1} \bigcap_{i \in \N} \big(r_{\frac{1}{i}}^v\big)^{-1}(B \cap Q) 
    \in \F_{\xi}
  \end{align*}
  is fulfilled.
\end{proof}

\begin{lemma}\label{lem:auxillary_bounds_sumH}
  Let Assumptions~\ref{ass:spaces}--\ref{ass:expectation} 
  be fulfilled. Further, let $f_{\xi_n}$ be an element of $C([0,T]; V_{\xi_n}^*)$ 
  almost surely for every $n \in \{1.\dots.N\}$. For $u' \in C^{\gamma}([0,T];H)$, 
  $\gamma \in (0,1]$, and a maximal 
  step size $h = \max_{i \in \{1,\dots,N\}} h_i$, it the follows that
  \begin{align*}    
    &\sum_{i=1}^{N} h_i \E_{\xi_i} \Big[ \Big\| f_{\xi_i}(t_i) - \Ai u(t_i) 
    - \frac{1}{h_i} \int_{t_{i-1}}^{t_i} \big(f(t) - A(t) u(t)\big) \diff{t} 
    \Big\|_H^2 \Big] \\
    &\leq 2 \sum_{i = 1}^{N} h_i \E_{\xi_i} \big[ \big\| f_{\xi_i}(t_i) - \Ai u(t_i)
    - (f(t_i) - A(t_i) u(t_i))  \big\|_H^2 \big] \\
    &\quad + 2 h^{2 \gamma} |u'|_{C^{\gamma}([0,T];H)}^2 T 
  \end{align*}
  and 
  \begin{align*}
    &\sum_{i=1}^{N} h_i \Big\| \E_{\xi_i} \Big[ 
    f_{\xi_i}(t_i) - \Ai u(t_i) 
    - \frac{1}{h_i} \int_{t_{i-1}}^{t_i} \big(f(t) - A(t) u(t)\big) \diff{t}  \Big] \Big\|_H^2\\
    &\leq h^{2\gamma} |u'|_{C^{\gamma}([0,T];H)}^2 T,
  \end{align*}
  where $|u'|_{C^{\gamma}([0,T];H)}$ is the H\"older semi-norm with values in $H$ of 
  the function $u'$.
\end{lemma}

\begin{proof} 
  To prove the first bound, we find that
  \begin{align*}
    & \sum_{i=1}^{N} h_i \E_{\xi_i} \Big[ \Big\| f_{\xi_i}(t_i) - \Ai u(t_i) 
    - \frac{1}{h_i} \int_{t_{i-1}}^{t_i} \big(f(t) - A(t) u(t)\big) \diff{t} 
    \Big\|_H^{2} \Big]\\
    &\leq 2 \sum_{i = 1}^{N} h_i \E_i \big[ \big\| f_{\xi_i}(t_i) - \Ai u(t_i)
    - (f(t_i) - A(t_i) u(t_i)) \big\|_H^{2}\big]\\
    &\quad + 2 \sum_{i = 1}^{N} h_i \E_i \Big[ \Big\| \frac{1}{h_i} 
    \int_{t_{i-1}}^{t_i} \big( f(t_i) - A(t_i) u(t_i) - (f(t) - A(t) u(t)) \big) \diff{t} \Big\|_H^{2}\Big]\\
    &\leq 2 \sum_{i = 1}^{N} h_i \E_i \big[ \big\| f_{\xi_i}(t_i) - \Ai u(t_i)
    - (f(t_i) - A(t_i) u(t_i))  \big\|_H^{2}\big]\\
    &\quad + 2 \sum_{i = 1}^{N} \frac{1}{h_i} \Big\| \int_{t_{i-1}}^{t_i} (u'(t_i) - u'(t)) 
    \diff{t} \Big\|_H^{2}.
  \end{align*}
  To further bound the last row, we apply H\"older's inequality and the regularity condition 
  $u' \in C^{\gamma}([0,T];H)$. We then find that
  \begin{align*}
    2 \sum_{i = 1}^{N} \frac{1}{h_i} \Big\| \int_{t_{i-1}}^{t_i} (u'(t_i) - u'(t)) 
    \diff{t} \Big\|_H^{2}
    &\leq 2 \sum_{i = 1}^{N} \int_{t_{i-1}}^{t_i} \| u'(t_i) - u'(t)\|_H^{2} \diff{t}\\
    &\leq 2 h^{2 \gamma} |u'|_{C^{\gamma}([0,T];H)}^{2} T.
  \end{align*}
  It remains to prove the second estimate of the lemma. 
  Recall that $\E_{\xi_i} \big[ f_{\xi_i}(t_i) \big] = f(t_i)$ and $\E_{\xi_i} \big[ \Ai u(t_i)\big] 
  = A(t_i) u(t_i)$ is fulfilled by Assumption~\ref{ass:expectation}. Using these equalities,
  it follows that
  \begin{align*}    
    &\sum_{i=1}^{N} h_i \Big\| \E_{\xi_i} \Big[ f_{\xi_i}(t_i) - \Ai u(t_i) 
    - \frac{1}{h_i} \int_{t_{i-1}}^{t_i} \big(f(t) - A(t) u(t)\big) \diff{t}  \Big] \Big\|_H^2\\
    &= \sum_{i=1}^{N} h_i \Big\| f(t_i) - A(t_i) u(t_i) - \frac{1}{h_i} 
    \int_{t_{i-1}}^{t_i} ( f(t) - A(t) u(t) ) \diff{t}\Big\|_H^2\\
    &\leq \sum_{i=1}^{N}  \int_{t_{i-1}}^{t_i} \| u'(t_i) - u'(t) \|_H^2 \diff{t}
    \leq h^{2\gamma} |u'|_{C^{\gamma}([0,T];H)}^2 T.
  \end{align*}
\end{proof}

\end{document}